\documentclass[12pt,reqno]{amsart}
\usepackage[margin=1in]{geometry}
\makeatletter
\def\section{\@startsection{section}{1}%
	\z@{.7\linespacing\@plus\linespacing}{.5\linespacing}%
	{\bfseries
		\centering
}}
\def\@secnumfont{\bfseries}
\makeatother
\usepackage{amsmath,amssymb,amsthm,graphicx,amsxtra, setspace}
\usepackage[utf8]{inputenc}
\usepackage{mathrsfs}
\usepackage{alltt}
\usepackage{relsize}
\usepackage{hyperref}
\usepackage{aliascnt}
\usepackage{tikz}
\usepackage{mathtools}
\usepackage{multicol}
\usepackage{upgreek}
\usepackage{graphicx,type1cm,eso-pic,color}
\allowdisplaybreaks

\newtheorem{theorem}{Theorem}[section]
\newtheorem{lemma}[theorem]{Lemma}

\newtheorem{assumption}[theorem]{Assumption}

\newtheorem{definition}[theorem]{Definition}
\newtheorem{example}[theorem]{Example}
\newtheorem{remark}[theorem]{Remark}

\renewcommand{\d}{\/\mathrm{d}\/}

\def\w{\textbf{W}^{\varepsilon}_{{\theta}^{\varepsilon}}}

\def\L{\mathbb{L}}
\def\X{\mathbb{X}}
\def\A{\mathrm{A}}

\def\F{\mathrm{F}}

\def\g{\mathbf{g}}

\def\C{\mathrm{C}}

\def\B{\mathrm{B}}

\def\y{\mathbf{y}}

\def\z{\mathbf{z}}
\def\p{\mathbf{p}}

\def\v{\mathbf{v}}
\def\w{\mathbf{w}}
\def\W{\mathrm{W}}

\def\f{\mathbf{f}}

\def\U{\mathrm{U}}

\def\u{\mathbf{u}}

\def\H{\mathbb{H}}

\newcommand{\R}{\mathbb{R}}

\renewcommand{\d}{\/\mathrm{d}\/}

\let\originalleft\left
\let\originalright\right
\renewcommand{\left}{\mathopen{}\mathclose\bgroup\originalleft}
\renewcommand{\right}{\aftergroup\egroup\originalright}

\newcommand{\Addresses}{{
		\footnote{
			
			\noindent \textsuperscript{1}Department of Mathematics, Indian Institute of Technology Roorkee-IIT Roorkee,
			Haridwar Highway, Roorkee, Uttarakhand 247667, INDIA.\par\nopagebreak
			\noindent  \textit{e-mail:} \texttt{manilfma@iitr.ac.in, maniltmohan@gmail.com.}
			
			\noindent \textsuperscript{*} Corresponding author.

			\textit{Key words:} Tidal dynamics system, Pontryagin's maximum principle, Optimal control.
			
			Mathematics Subject Classification (2010): 49J20, 35Q35, 49K20.

}}}

\begin{document}

	\title[Optimal Control of the  2D Tidal Dynamics System]{First Order Necessary Conditions of Optimality for the Two Dimensional Tidal Dynamics System\Addresses}
	\author[M. T. Mohan]
	{Manil T. Mohan\textsuperscript{1*}}

	\maketitle
	
	\begin{abstract}
		In this work, we consider the two dimensional tidal dynamics equations in a bounded domain and address some optimal control problems like total energy minimization, minimization of dissipation of energy of the	flow, etc. We also examine an another interesting control problem which is similar to that of the data assimilation problems in meteorology of obtaining unknown initial data, when the system under consideration is the tidal dynamics, using optimal control techniques. For these cases, different  distributed optimal control problems are formulated as the minimization of suitable cost functionals subject to the controlled two dimensional tidal dynamics system. The existence of an optimal control as well as the first order necessary conditions of optimality for such systems is established and the optimal control is characterized via  adjoint variable.  We also establish the uniqueness of optimal control in small time interval.   
	\end{abstract}

\section{Introduction}\label{sec1}\setcounter{equation}{0} 
Many mathematical developments in the infinite dimensional nonlinear system theory and partial differential equations awarded a new dimension to the control theory of fluid dynamics models (cf. \cite{FART,SDMTS,fursikov,gunzburger,lions,raymond,sritharan} etc). Optimal control theory of fluid dynamic equations has been one of the major research areas of applied mathematics with  a good number of applications in Oceanography, Geophysics, Engineering and Technology (see for example \cite{VB,fursikov,gunzburger, sritharan}, etc). Controlling fluid flow and turbulence inside a flow in a given physical domain by means of body forces, boundary data, temperature, initial data, etc,  is an interesting control problem in fluid mechanics. Ocean tides have been investigated by many mathematicians and physicists,	starting from  Galileo Galilei, Isaac Newton etc (see \cite{GaG,IN}). The ocean tide informations are heavily used in  the geophysical areas such as Earth tides, the elastic properties of the Earth's crust, tidal variations of gravity, and in calculating the orbits of artificial satellites used for space exploration etc (cf. \cite{MK,MK1}). Laplace rearranged the rotating shallow water equations into the system that underlies the tides and is known as the \emph{Laplace tidal equations}. By taking the shallow water model on a rotating sphere, which is a slight generalization of the Laplace model, the tidal dynamics model considered in \cite{BAK,MK,MK1} is obtained. We refer the readers to \cite{MK,MK1,PJ} etc,  for an extensive study on the recent progress in this field.  In this article, we consider the controlled two dimensional tidal dynamics equations in bounded domains (see \eqref{0} and \eqref{1} below) and study the optimal control problems including total energy minimization problem (the tracking problem), minimization of dissipation of energy of the flow, an initial value optimization problem, which is similar to the data assimilation problem in meteorology,  etc. We establish the first order necessary conditions of optimality and characterize the optimal control via adjoint variable. 

Let us now describe the two dimensional tidal dynamics equations with a distributed control. Let $\Omega$ be a bounded subset of $\mathbb{R}^2$ with  smooth boundary conditions.  That is, $\Omega$ is a horizontal ocean basin, where tides are induced over the time interval $[0, T ]$. The boundary contour $\partial\Omega$ is composed of two disjoint parts: a solid part  $\Gamma_1$, coinciding with the edge of the continental and island shelves, an open boundary $\Gamma_2$. Let us assume that sea water is incompressible and the vertical velocities are small compared with the horizontal velocities, and hence we are able to exclude acoustic waves. Also long waves, including tidal waves, are stood out from the family of gravitational oscillations. Moreover, in order to reduce computational difficulties, we assume that the Earth is absolutely rigid, and the gravitational field of the Earth is not affected by movements of ocean tides. Also, we ignore the effect of the atmospheric tides on the ocean tides and the effect of curvature of the surface of the Earth on horizontal turbulent friction. Under these commonly used assumptions, we consider the following controlled tidal dynamics model (cf. \cite{RGG,BAK,MK,MK1,Manna,yin}, etc):
\begin{equation}\label{0}
\left\{
\begin{aligned}
\frac{\partial\mathbf{w}}{\partial t}+l\mathbf{k}\times\mathbf{w}+g\nabla\zeta+\frac{r}{h}|\mathbf{w}|\mathbf{w}-\kappa_h\Delta\mathbf{w}&=\mathbf{g}+\U,\ \text{ in }\Omega\times(0,T), \\
\frac{\partial\zeta}{\partial t}+\mathrm{div}(h\mathbf{w})&=0,\ \text{ in }\Omega\times(0,T), \\
\mathbf{w}&=\mathbf{w}^0,  \ \text{ on }\partial\Omega\times[0,T],\\ \mathbf{w}(0)&=\mathbf{w}_0,\ \zeta(0)=\zeta_0,\ \text{ in }\Omega,\\
\mathbf{w}^0=\mathbf{0},\text{ on }\Gamma_1,&\  \text{ and } \ \int_0^T\int_{\Gamma_2}h\mathbf{w}^0\cdot \mathbf{n}\d\Gamma_2\d t=0,
\end{aligned}
\right.
\end{equation}
where  $\mathbf{w}(x,t)=(w_1(x,t),w_2(x,t))\in\mathbb{R}^2$, the horizontal transport vector, is the averaged integral of the velocity vector over the vertical axis, $l = 2\rho\cos\theta$ is the Coriolis parameter, where $\rho$ is the angular velocity of the Earth rotation and $\theta$ is the colatitude, $\mathbf{k}$ is a unit vector oriented vertically upward, $\mathbf{k}\times\w=(-w_2,w_1)$, $g$ is the free fall acceleration,  $r$ is the bottom friction factor,  $\kappa_h$ is the horizontal turbulent viscosity coefficient, $\mathbf{n}$ is the unit outward normal to the boundary $\Gamma_2$. The scalar  $\zeta(x,t)\in\mathbb{R}$ is the deviations of free surface with respect to the ocean bottom (i.e.,  $\zeta=\zeta_s-\zeta_b$, where $\zeta_s$ is the surface and $\zeta_b$ is the shift of the ocean bottom),  $\mathbf{g}=\gamma_Lg\nabla\zeta^+$ is the known tide-generating force with $\gamma_L$ is the Love factor approximately equal to $0.7$ and $\zeta^+$ is the height of the static tide.  Also $\U$ is the \emph{distributed control} (an external forcing) acting on the system.  The tidal dynamics system \eqref{0} is of hyperbolic-parabolic type.   Some practical applications of controlling tidal dynamics  has been described in \cite{RMo}. The water level control is simulated by means of gate operations acting at the open mouth of the tidal basin. The author  studied the tidal fluctuations in the basin under some gate operations, which do not require complete closure of the basin mouth, but only a variable reduction of its width (see \cite{RMo} for more details).

Let us now discuss about  the boundary conditions satisfied by the averaged velocity field. Remember that the contour $\partial\Omega$ consists of two parts, a solid part $\Gamma_1$ coinciding with the shelf edge and the open boundary $\Gamma_2$. The function $\mathbf{w}^0(x,t)$ is a known function  on the boundary. This impermeability condition is given on the solid part of the boundary, that is, the restriction $\mathbf{w}^0|_{\Gamma_1}= \mathbf{0}$ is the no-slip boundary condition on the shoreline, and $\int_0^T\int_{\Gamma_2}h\mathbf{w}^0\cdot\mathbf{n}\d\Gamma_2\d t=0$, follows from the mass conservation law. In \eqref{0}, $h(x)$ is the vertical scale of motion, that is, the depth of the calm sea at $x$ in the region $\Omega$  and we assume that it is a continuously differentiable function nowhere becoming zero,  so that 
\begin{align}
\label{hes}
0<\lambda= \min_{x\in\Omega} h(x),\quad \mu= \max_{x\in \Omega}h(x),\quad \max_{x\in \Omega}|\nabla h(x)|\leq M,\end{align}
where $M$ is a positive constant, which equals to zero at a constant ocean depth. The initial datum  $\mathbf{w}_0(x)$ and $\zeta_0(x)$ are given. In particular, $\mathbf{w}_0$ and $\zeta_0$ can be set equal to zero, which indicates the fact that initially the ocean is at rest.  The unique global solvability results of the system \eqref{0} is obtained by simplifying the non-homogeneous boundary value problem to a homogeneous Dirichlet boundary value problem (see \cite{IpM,MK,MK1}, etc for more details).

In order to simplify the non-homogeneous boundary value problem to a homogeneous Dirichlet boundary value problem, we set
\begin{align*}
\mathbf{u}(x,t)=\mathbf{w}(x,t)-\mathbf{w}^0(x,t),
\end{align*}
and 
\begin{align*}
\xi(x,t)=\zeta(x,t)+\int_0^t\mathrm{div}(h(x)\mathbf{w}^0(x,s))\d s,
\end{align*}	
which are referred to as the \emph{tidal flow} and the \emph{elevation}. The full flow $\mathbf{w}^0$, which is given a priori on the boundary $\partial\Omega$, has been extended to the whole domain $\Omega\times[0,T]$ as a smooth function and still denoted by $\mathbf{w}^0$. Intuitively, one can see that a solution of the Laplace equation (in $\mathrm{Q}_T:=\Omega\times[0,T]\subset\R^3$) with $\w^0$ as the boundary condition is  smooth.   Whitney in \cite{HW} solved the problem of extending functions from a domain with a sufficiently smooth boundary to the whole space while preserving continuity of the partial derivatives. Thus, using Whitney's construction, there is a $\C^{\infty}$ extension, by writing $\mathrm{Q}_T$ as a union of Whitney cubes (see Theorem 3, Chapter I, \cite{EMS} also) and using $\C^{\infty}$  partition of unity. With the above substitution, the controlled tidal dynamics system \eqref{0} can be written as (see \eqref{a1} below for an abstract form):

\begin{equation}\label{1}
\left\{
\begin{aligned}
\frac{\partial \mathbf{u}(t)}{\partial t}+\A\mathbf{u}(t)+\B(\mathbf{u}(t))+\nabla \xi (t)&=\mathbf{f}(t)+\U(t),\ \text{ in }\ \Omega\times(0,T),\\
\frac{\partial \xi(t)}{\partial t}+\mathrm{div}(h \mathbf{u}(t))&=0, \ \text{ in }\ \Omega\times(0,T),\\
\mathbf{u}(t)&=\mathbf{0},\ \text{ on }\ \partial\Omega\times[0,T],\\
\mathbf{u}(0)&=\mathbf{u}_0,\ \xi(0)=\xi_0, \ \text{ in }\ \Omega,
\end{aligned}
\right.
\end{equation}
where we scaled $g$ to unity. For $\U=\mathbf{0}$, we call the system \eqref{1} as an uncontrolled tidal dynamics system. In \eqref{1}, $\A$ denotes the matrix operator:
\begin{align}\label{2}
\A := 
\left(\begin{array}{cc} -\alpha\Delta & -\beta\\ \beta &  -\alpha\Delta \end{array}\right),
\end{align}
where $\Delta$ is the Laplacian operator, $\alpha:=\kappa_h$, $\beta:=2\rho\cos\theta$ are positive constants, $\B$ denotes the nonlinear vector operator,
\begin{align}\label{3}
\B(\mathbf{u}):=\gamma(x)|\mathbf{u}+\mathbf{w}^0|(\mathbf{u}+\mathbf{w}^0).
\end{align} 
The function $\gamma(x):=\frac{r}{h(x)}$ is a strictly positive smooth function.
The function $\f$ and initial datum $\mathbf{u}_0$ and $\xi_0$ are  given by 
\begin{equation}\label{5}
\left\{
\begin{aligned}
\f&=\mathbf{g}-\frac{\partial\mathbf{w}^0}{\partial t}+\nabla\int_0^t\mathrm{div}(h\mathbf{w}^0)\d s+\kappa_h\Delta\mathbf{w}^0-l\mathbf{k}\times\mathbf{w}^0,\\
\mathbf{u}_0(x)&=\mathbf{w}_0(x)-\mathbf{w}^0(x,0), \\ \xi_0(x)&=\zeta_0(x).
\end{aligned}
\right.
\end{equation}

Let us now discuss  about the solvability results available  in the literature for the system \eqref{0}. The existence and uniqueness of a weak solution for the tidal dynamic equations (see systems \eqref{0} or \eqref{1}) in bounded domains has been obtained in \cite{IpM,MK,MK1}, using a standard Galerkin approximation technique and compactness arguments. The authors in \cite{Manna} obtained similar results for the deterministic tidal dynamics system using a global monotonicity property of the linear and nonlinear operators (see Lemma \ref{lem23} below). The existence of a periodic solution for the tidal dynamics problem in two dimensional finite domains is obtained in \cite{RGG}. The existence and uniqueness of weak and strong solutions of the stationary tidal dynamic equations in bounded and unbounded domains is obtained in \cite{MTM}. Furthermore, the author in \cite{MTM} established a uniform Lyapunov  stability of the steady state solution. The authors in \cite{VIA} developed a numerical approach  for solving the tidal dynamics problem, based on the splitting methods and the optimal control theory. The global solvability results for stochastic perturbations in bounded and unbounded domains, and asymptotic analysis of solutions  have been established in \cite{AMM,HSMB,Manna,SSKB,yin} etc.

The control of water level in tidal basins like harbours, bays, lagoons and tidal inlets, etc is an interesting problem, which has  many applications in the practical sense  such as defense from storm surges, creation of water reservoirs and conversion of tidal energy (see \cite{RMo} for more details). The author in \cite{RMo} applied Pontryagin's maximum principle  in optimizing the sea level in a small basin  induced by a tidal component, where the equations describing the dynamics of the basin have been simplified.  Pontryagin's maximum principle for the state constrained optimization problem using Ekeland's variational principle is established in \cite{MTM2}.  The dynamic programming method and feedback analysis for an optimal control of the 2D tidal dynamics system is carried out in \cite{MTM1}. The authors in \cite{AMM} formulated a martingale problem of Stroock and Varadhan associated to an initial value control problem and established the existence of optimal controls.  Some optimal control problems in tidal power generation and related problems are considered in \cite{NRCB,RyR,RyR1,ZYJM}, etc. 

The rest of the paper is organized as follows. In the next section, we give the necessary functional setting needed to obtain the global solvability results of the system \eqref{1}. We also consider the corresponding linearized system in the same section and establish the existence and uniqueness of a global weak solution (Theorem \ref{linear}). A distributed optimal control problem  as the minimization of a suitable cost functional subject to the controlled tidal dynamics system \eqref{1} is formulated in section \ref{sec3}. The existence of an optimal triplet (Theorem \ref{optimal}) as well as the first order necessary conditions of optimality is also established in this section (Theorem \ref{main}). The optimal control is characterized using the adjoint variable and the solvability of the adjoint system is also discussed (Theorem \ref{thm3.4}). We also establish the uniqueness of optimal control in small time interval (Theorem \ref{thm3.8}).    Similar results for a different optimization problem, namely initial data optimization problem (a problem similar to the data assimilation problems of meteorology) is also obtained in the same section (Theorems \ref{data}).

\section{Mathematical Formulation}\label{sec2} \setcounter{equation}{0} 
In this section, we explain the necessary  function spaces needed to obtain the global solvability results and provide the global existence and uniqueness of weak solution of the uncontrolled system \eqref{1}. Also, we provide an insight into the global solvability of the corresponding linearized problem.

\subsection{Functional setting} 
For $p\geq 1$, we denote by $\mathrm{L}^p(\Omega):=\mathrm{L}^p(\Omega;\mathbb{R})$, for the equivalence class of measurable real-valued functions  for which the $p^{\text{th}}$ power of the absolute value is Lebesgue integrable. We know that $\mathrm{L}^2(\Omega)$ is a Hilbert space, and the norm and inner product in $\mathrm{L}^2(\Omega)$ are denoted by $\|\cdot\|_{\mathrm{L}^2}$ and $(\cdot,\cdot)_{\mathrm{L}^2}$. We define $\mathbb{L}^p(\Omega):=\mathrm{L}^p(\Omega; \mathbb{R}^2)$ as the Banach space of Lebesgue measurable $\mathbb{R}^2$-valued, $p$-integrable functions on $\Omega$  with the norm:
$$\|\mathbf{u}\|_{\mathbb{L}^p}:= \left(\int_\Omega|\mathbf{u}(x)|^p\mathrm{d}x\right)^{1/p}.$$
For $p=2,\,\mathbb{L}^2(\Omega):=\mathrm{L}^2(\Omega; \mathbb{R}^2)$  is a Hilbert space equipped with the inner product given by 
$$ (\mathbf{u}, \v)_{\mathbb{L}^2} :=\int_\Omega\mathbf{u}(x)\cdot \v(x)\mathrm{d}x,\  \mathbf{u},\v  \in \mathbb{L}^2(\Omega).$$
Then, the norm on $\mathbb{L}^2(\Omega)$ is defined by  $$\|\mathbf{u}\|_{\mathbb{L}^2}:=(\mathbf{u},\mathbf{u})_{\mathbb{L}^2}=\left(\int_\Omega|\mathbf{u}(x)|^2\mathrm{d}x\right)^{1/2}.$$
Let $\mathbb{H}^1(\Omega):=\mathrm{H}^1(\Omega;\mathbb{R}^2)$ denotes the Sobolev space $\mathbb{W}^{1,2}(\Omega):=\mathrm{W}^{1,2}(\Omega;\mathbb{R}^2)$ with the norm defined by 
$$\|\mathbf{u}\|_{\mathbb{H}^1}^2:=\|\mathbf{u}\|_{\mathbb{L}^2}^2+\|\nabla \mathbf{u}\|_{\mathbb{L}^2}^2, \ \text{ for all } \ \mathbf{u} \in \mathbb{H}^1(\Omega).$$ 
We also let $\mathbb{H}_0^1(\Omega):=\mathrm{H}_0^1(\Omega; \mathbb{R}^2)$ to be the closure  of $\mathrm{C}_c^\infty(\Omega;\mathbb{R}^2)$ in $\mathbb{H}^1(\Omega)$ norm, where $\mathrm{C}_c^\infty(\Omega;\mathbb{R}^2)$ is the space of all infinitely differentiable functions with compact support in $\Omega$. Then $\mathbb{H}_0^1(\Omega)$ is also a Sobolev space under the induced norm. Since $\Omega$ is a bounded domain, in view of the Poincar\'e inequality, that is, $$\|\mathbf{u}\|_{\mathbb{L}^2}\leq C_{\Omega}\|\nabla \mathbf{u}\|_{\mathbb{L}^2},$$ the norms $\|\nabla \mathbf{u}\|_{\mathbb{L}^2(\Omega)}$ and $\|\mathbf{u}\|_{\mathbb{H}^1}$  are equivalent in $\mathbb{H}_0^1(\Omega)$. Then, we know that   $$ \|\mathbf{u}\|_{\mathbb{H}_0^1}:=\|\nabla \mathbf{u}\|_{\mathbb{L}^2}=\left(\int_\Omega|\nabla \mathbf{u}(x)|^2\mathrm{d}x\right)^{1/2},$$ defines a norm on  $\mathbb{H}_0^1(\Omega)$, which is equivalent to the usual $\H^1(\Omega)$ norm.  Moreover,  $\H_0^1(\Omega)$ is a Hilbert space with inner product: 
$$(\mathbf{u},\v)_{\mathbb{H}_0^1}:=(\nabla\mathbf{u},\nabla\v)_{\mathbb{L}^2}=\int_{\Omega}\nabla\mathbf{u}(x)\cdot\nabla\v(x)\d x.$$

We denote the dual of $\mathbb{H}_0^1(\Omega)$ by $(\mathbb{H}_0^1(\Omega))'=\mathbb{H}^{-1}(\Omega)$. The induced duality between the spaces $\mathbb{H}_0^1(\Omega)$ and  $\mathbb{H}^{-1}(\Omega)$ is denoted by
$\langle\cdot{,}\cdot\rangle$.  Then, we have the following continuous and dense embedding:
$$\mathbb{H}_0^1(\Omega)\subset\mathbb{L}^2(\Omega)\equiv(\mathbb{L}^2(\Omega))'\subset\mathbb{H}^{-1}(\Omega).$$ The   embedding $\mathbb{H}_0^1(\Omega)\subset\mathbb{L}^2(\Omega)$ is also compact, since $\Omega$ is bounded. Using the Gelfand triple $(\mathbb{H}_0^1(\Omega), \mathbb{L}^2(\Omega), \mathbb{H}^{-1}(\Omega))$, we may consider $\nabla$ or $\Delta$ as a linear map from  $\mathbb{L}^2(\Omega)$ or $\mathbb{H}_0^1(\Omega)$ into $\mathbb{H}^{-1}(\Omega)$ respectively.

We next give the well known inequality due to Ladyzhenskaya (see Lemma 1 and 2, Chapter 1, \cite{OAL}), which is used in the paper quite frequently. 
\begin{lemma}[Ladyzhenskaya inequality]\label{lady}
	For $\mathbf{u}\in\ \mathrm{C}_0^{\infty}(\Omega;\mathbb{R}^n), n = 2, 3$, there exists a constant $C>0$ such that
	\begin{align}\label{lady1}
	\|\mathbf{u}\|_{\mathbb{L}^4}\leq C^{1/4}\|\mathbf{u}\|_{\mathbb{L}^2}^{1-\frac{n}{4}}\|\nabla\mathbf{u}\|_{\mathbb{L}^2}^{\frac{n}{4}},\text{ for } n=2,3,
	\end{align}
	where $C=2,4$, for $n=2,3$ respectively. 
\end{lemma}
Thus, for $n=2$, we have 
\begin{align}\label{8}
\|\mathbf{u}\|_{\mathbb{L}^4}\leq 2^{1/4}\|\mathbf{u}\|_{\mathbb{L}^2}^{1/2}\|\nabla\mathbf{u}\|_{\mathbb{L}^2}^{1/2}\leq 2^{1/4} C_{\Omega}^{1/2}\|\nabla\mathbf{u}\|=\widetilde{C}_{\Omega}\|\mathbf{u}\|_{\mathbb{H}_0^1},
\end{align}
where $\widetilde{C}_{\Omega}=2^{1/4}C_{\Omega}$ and we also used the Poincar\'e inequality. Thus, we obtain a continuous embedding  $\mathbb{H}_0^1(\Omega)\hookrightarrow \mathbb{L}^4(\Omega)$. 

\subsection{Linear operator} Let us define the  non-symmetric bilinear form:
\begin{align*}
a(\mathbf{u},\v):=\alpha[(\nabla\mathbf{u}_1, \nabla\v_1)+(\nabla\mathbf{u}_2, \nabla\v_2)]+\beta[(\mathbf{u}_1,\v_2)-(\mathbf{u}_2, \v_1)],
\end{align*}
where $\mathbf{u}=(\mathbf{u}_1,\mathbf{u}_2), \v=(\v_1,\v_2)$. If $\mathbf{u}$ has a smooth second order derivatives, then
$$a(\mathbf{u},\v)=(\A\mathbf{u}, \v)_{\mathbb{L}^2}, \ \text{ for all }\ \v\in \mathbb{H}_0^1(\Omega).$$ We consider $\A\mathbf{u}=-\alpha \Delta \mathbf{u}+\beta \mathbf{k}\times \mathbf{u}$ and an integration by parts twice yields (using the fact that $\mathbf{u}\big|_{\partial\Omega}=\mathbf{0}$)
\begin{align*}
(-\alpha \Delta \mathbf{u}+\beta \mathbf{k}\times \mathbf{u},\v)_{\mathbb{L}^2}&=\alpha(\nabla\mathbf{u},\nabla\v)_{\mathbb{L}^2}+\beta(\mathbf{k}\times\mathbf{u},\v)_{\mathbb{L}^2}\\&=(\mathbf{u},-\alpha\Delta\v-\beta\mathbf{k}\times\v)_{\mathbb{L}^2},
\end{align*}
and hence $(\A\mathbf{u},\v)_{\mathbb{L}^2}\neq (\mathbf{u},\A\v)_{\mathbb{L}^2}$, so that $\A$ is not symmetric. 
The bilinear form $a(\cdot,\cdot)$ is continuous and coercive in $\mathbb{H}_0^1(\Omega)$, that is,
\begin{align}
|a(\mathbf{u},\v)|&\leq C_a\|\mathbf{u}\|_{\mathbb{H}_0^1}\|\v\|_{\mathbb{H}_0^1}, \ \text{ for all }\ \mathbf{u}, \v\in \mathbb{H}_0^1(\Omega),\label{6}\\ 
(\A\mathbf{u},\mathbf{u})_{\mathbb{L}^2}&=a(\mathbf{u},\mathbf{u})=\alpha\|\nabla \mathbf{u}\|_{\mathbb{L}^2}^2=\alpha\|\mathbf{u}\|_{\mathbb{H}_0^1}^2,\label{7}
\end{align}
for some positive constant $C_a$. By means of the Gelfand triple we may consider $\A$, given by \eqref{2}, as a mapping from $\mathbb{H}_0^1(\Omega)$ into its dual $\mathbb{H}^{-1}(\Omega)$, so that $\langle \A\mathbf{u},\mathbf{u}\rangle=\alpha\|\nabla \mathbf{u}\|_{\mathbb{L}^2}^2=\alpha\|\mathbf{u}\|_{\mathbb{H}_0^1}^2$.
\subsection{Nonlinear operator}
As discussed in \eqref{3}, we define the nonlinear operator $\B(\cdot)$ by 
$$\v\mapsto\B(\v):=\gamma(x)|\v+\mathbf{w}^0|(\v+\mathbf{w}^0).$$   In the following, $\mathcal{L}(\mathrm{H};\mathrm{K})$ denotes the space of all bounded linear operators from $\mathrm{H}$ to $\mathrm{K}$.

\begin{lemma}[\cite{Manna,SSKB,MTM}]\label{lem2.3}
	The operator $\B$ has the following properties:  For all $\mathbf{u},\v,\mathbf{w}^0\in\mathbb{L}^4(\Omega)$, we have 
	\begin{enumerate}
		\item [(i)] $\|\B(\mathbf{u})\|_{\mathbb{L}^2}\leq \frac{r}{\lambda}(\|\mathbf{u}\|_{\mathbb{L}^4}+\|\mathbf{w}^0\|_{\mathbb{L}^4})^2$,  \item [(ii)] $\B(\cdot)$ is a nonlinear continuous operator from $\mathbb{H}_0^1(\Omega)$ into $\mathbb{H}^{-1}(\Omega)$,
		\item [(iii)] $(\B(\mathbf{u})-\B(\v), \mathbf{u}-\v)_{\mathbb{L}^2} \geq 0$,
		\item [(iv)] $(\B(\mathbf{u}), \mathbf{u})_{\mathbb{L}^2}\geq  -\frac{r}{2\lambda}(\|\mathbf{w}^0\|_{\mathbb{L}^4}^4+\|\mathbf{u}\|_{\mathbb{L}^2}^2)$,
		\item [(v)] $\|\B(\mathbf{u})-\B(\v)\|_{\mathbb{L}^2}\leq \frac{r}{2\lambda}(\|\mathbf{u}\|_{\mathbb{L}^4}+\|\v\|_{\mathbb{L}^4}+\|\mathbf{w}^0\|_{\mathbb{L}^4})\|\mathbf{u}-\v\|_{\mathbb{L}^4}$,
		\item [(vi)] {the operator $\B(\cdot)$ is Fr\'echet differentiable with the Fr\'echet derivative $\B'(\mathbf{u})=2\gamma|\mathbf{u}+\mathbf{w}^0|\in\mathcal{L}(\mathbb{L}^4;\mathbb{L}^2)$ and $(\B'(\mathbf{u})\v,\v)_{\mathbb{L}^2}\geq 0$,} 
		\item [(vii)] $\|\B'(\mathbf{u})\v\|_{\mathbb{H}^{-1}}\leq \frac{2C_{\Omega}r}{\lambda}\left(\|\u\|_{\L^4}+\|\w^0\|_{\L^4}\right)\|\v\|_{\L^4}$.
	\end{enumerate}
\end{lemma}
\begin{proof}
	We prove only (vii). For $\w\in\H_0^1(\Omega)$, we consider
	\begin{align}
	|	\langle\B'(\mathbf{u})\v,\w\rangle |&\leq\|\B'(\mathbf{u})\v\|_{\mathrm{L}^2}\|\w\|_{\mathrm{L}^2}\leq \frac{2C_{\Omega}r}{\lambda}\|\u+\w^0\|_{\L^4}\|\v\|_{\L^4}\|\w\|_{\H_0^1}\nonumber\\&\leq \frac{2C_{\Omega}r}{\lambda}\left(\|\u\|_{\L^4}+\|\w^0\|_{\L^4}\right)\|\v\|_{\L^4}\|\w\|_{\H_0^1},
	\end{align}
	so that we have $\|\B'(\mathbf{u})\v\|_{\mathbb{H}^{-1}}\leq \frac{2C_{\Omega}r}{\lambda}\left(\|\u\|_{\L^4}+\|\w^0\|_{\L^4}\right)\|\v\|_{\L^4}$. 
\end{proof}
The estimates \eqref{7} and (iii) easily imply the following: 
\begin{lemma}\label{lem23}
	The operator $\mathrm{F}(\mathbf{u}):=\A\mathbf{u}+\B(\mathbf{u})-\f$ is a globally monotone operator, that is, $$\langle\F(\mathbf{u})-\F(\v),\mathbf{u}-\v\rangle \geq 0, \ \text{ for all }\ \mathbf{u},\v\in\mathbb{H}_0^1(\Omega).$$ 
\end{lemma}
\subsection{Global existence and uniqueness} Let us now give the definition of weak solution and discuss the solvability results available in the literature for the uncontrolled system \eqref{1} (see \cite{MK,MK1,Manna,MTM}, etc for more details). 
\begin{definition}\label{weakd}
	The pair $$(\mathbf{u},\xi)\in(\mathrm{C}([0,T];\mathbb{L}^2(\Omega))\cap\mathrm{L}^2(0,T;\mathbb{H}_0^1(\Omega)))\times\mathrm{C}([0,T];\mathrm{L}^2(\Omega)),$$ with $$(\partial_t\mathbf{u},\partial_t \xi)\in\mathrm{L}^2(0,T;\mathbb{H}^{-1}(\Omega))\times\mathrm{L}^2(0,T;\mathrm{L}^2(\Omega)),$$  is called a \emph{weak solution} to the system(\ref{1}), if for $\f\in\mathrm{L}^2(0,T;\mathbb{H}^{-1}(\Omega))$, $(\mathbf{u}_0,\xi_0)\in\mathbb{L}^2(\Omega)\times\mathrm{L}^2(\Omega)$ and $(\v,\eta)\in\mathbb{H}_0^1(\Omega)\times\mathrm{L}^2(\Omega)$, $(\mathbf{u},\xi)$ satisfies:
	\begin{equation}\label{3.13}
	\left\{
	\begin{aligned}
	\langle\partial_t\mathbf{u}(t)+\A\mathbf{u}(t)+\B(\mathbf{u}(t))+\nabla \xi (t),\v\rangle&=\langle\f(t),\v\rangle,\\
	(\partial_t\xi(t)+\mathrm{div}(h\mathbf{u}(t)),\eta(t))_{\mathrm{L}^2}&=0,\\
	\lim_{t\downarrow 0}\int_{\Omega}\mathbf{u}(t,x)\v(x)\d x&=\int_{\Omega}\mathbf{u}_0(x)\v(x)\d x, \\ \lim_{t\downarrow 0}\int_{\Omega}\xi(t,x)\eta(x)\d x&=\int_{\Omega}\xi_0(x)\eta(x)\d x,
	\end{aligned}
	\right.
	\end{equation}
	and the energy equality 
	\begin{align}
	\frac{\d}{\d t}\left(\|\sqrt{h}\mathbf{u}(t)\|^2_{\mathbb{L}^2}+\|\xi(t)\|_{\mathrm{L}^2}^2\right)+2\langle\mathrm{F}(\mathbf{u}(t)),h\mathbf{u}(t)\rangle =0, 
	\end{align}
	for a.e. $t\in[0,T]$. 
\end{definition}

\begin{theorem}[Existence and uniqueness of weak solution, Chapter 2, \cite{MK}, Propositions 3.6, 3.7 \cite{Manna}]\label{weake} Let  $(\mathbf{u}_0,\xi_0)\in\mathbb{L}^2(\Omega)\times\mathrm{L}^2(\Omega)$ be given.  For $\f\in\mathrm{L}^2(0,T;\mathbb{H}^{-1}(\Omega))$, there exists a unique weak solution $(\mathbf{u},\xi)$ to the uncontrolled system (\ref{1}) satisfying 
	\begin{align}\label{energy}
	&\sup_{t\in[0,T]}(\|\u(t)\|_{\mathbb{L}^2}^2+\|\xi(t)\|_{\mathrm{L}^2}^2)+\alpha\int_0^T\|\nabla\u(t)\|_{\mathbb{L}^2}^2\d t\nonumber\\&\leq \left(\|\u_0\|_{\mathbb{L}^2}^2+\|\xi_0\|_{\mathrm{L}^2}^2+\frac{r}{\lambda}\int_0^T\|\w^0(t)\|_{\mathbb{L}^4}^4\d t+\int_0^T\|\f(t)\|_{\H^{-1}}^2\d t\right)e^{KT},
	\end{align}
	where $K=\max\left\{1+M+\frac{r}{\lambda},\frac{2}{\alpha}(1+\mu^2)+M\right\}$. 
\end{theorem}

\begin{remark}
	(1). {In order to obtain the regularity $\f\in\mathrm{L}^2(0,T;\mathbb{H}^{-1}(\Omega))$, one has to take $\g\in\mathrm{L}^2(0,T;\H^{-1}(\Omega))$ and $\w^0\in\mathrm{L}^{\infty}(0,T;\L^2(\Omega))\cap\mathrm{L}^2(0,T;\H^1(\Omega))$ with $\partial_t\w^0\in\mathrm{L}^2(0,T;\H^{-1}(\Omega))$ (see \eqref{5}). }
	\vskip 0.2 cm
	(2).  Using Ladyzhenskaya's inequality (see \eqref{8}), we obtain 
	\begin{align}\label{29}
	\int_0^T\|\u(t)\|_{\L^4}^4\d t\leq 2\sup_{t\in[0,T]}\|\u(t)\|_{\L^2}^2\int_0^T\|\nabla\u(t)\|_{\L^2}^2\d t<+\infty,
	\end{align}
	and hence $\u\in\mathrm{L}^4(0,T;\L^4(\Omega))$. 
\end{remark}

\subsection{The linearized system} Let us linearize the equations \eqref{1} around $(\widehat{\u}, \widehat{\xi})$ which is the \emph{unique weak solution} of system (\ref{1}) with the control term $\U=\mathbf{0}$ (uncontrolled system), external forcing $\widetilde{\mathbf{g}}$, and initial datum $\widehat{\u}_0$ and $\widehat{\xi}_0$ are such that $\widetilde{\mathbf{g}}\in\mathrm{L}^2(0,T;\mathbb{H}^{-1}(\Omega))$ and $(\mathbf{u}_0,\xi_0)\in\mathbb{L}^2(\Omega)\times\mathrm{L}^2(\Omega)$. We consider the following linearized system in the abstract formulation:
\begin{equation}\label{4}
\left\{
\begin{aligned}
\frac{\partial \mathfrak{w}(t)}{\partial t}+\A\mathfrak{w}(t)+\B'(\widehat{\mathbf{u}}(t))\mathfrak{w}(t)+\nabla \eta (t)&=\widehat{\mathbf{g}}(t)+\U(t),\ \text{ in }\ \H^{-1}(\Omega),\\
\frac{\partial \eta(t)}{\partial t}+\mathrm{div}(h \mathfrak{w}(t))&=0, \ \text{ in }\ \mathrm{L}^2(\Omega),\\
(\mathfrak{w}(0), \eta(0))&=(\mathfrak{w}_0,\eta_0)\in\L^2(\Omega)\times\mathrm{L}^2(\Omega), 
\end{aligned}
\right.
\end{equation}
for a.e. $t\in[0,T]$,  where $\widehat{\g}={\g}-\widetilde{\g}$, $\B'(\widehat{\mathbf{u}})=2\gamma|\widehat{\u}+\w^0|$ and $\mathfrak{w}(x,t)=\mathbf{0}$, for a.e. $(x,t)\in\partial\Omega\times[0,T]$.  Let us now prove an a-priori energy estimate satisfied by the system \eqref{4}. We take an inner product with $\mathfrak{w}(\cdot)$ to the first equation in \eqref{4} to obtain 
\begin{align}\label{216}
&\frac{1}{2}\frac{\d }{\d t}\|\mathfrak{w}(t)\|_{\mathbb{L}^2}^2+\alpha\|\nabla\mathfrak{w}(t)\|_{\mathbb{L}^2}^2\nonumber\\&=-(\B'(\widehat{\mathbf{u}}(t))\mathfrak{w}(t),\mathfrak{w}(t))_{\mathbb{L}^2}-\langle\nabla \eta (t),\mathfrak{w}(t)\rangle+\langle\widehat{\mathbf{g}}(t),\mathfrak{w}(t)\rangle+\langle\U(t),\mathfrak{w}(t)\rangle. 
\end{align}
Using an integration by parts,  the fact that $\|\mathrm{div\ }\u\|_{\mathrm{L}^2}\leq \sqrt{2}\|\nabla\u\|_{\mathbb{L}^2}$, Cauchy-Schwarz's  and Young's inequalities, we estimate the final three terms from the right hand side of the inequality \eqref{216} as 
\begin{align} 
-&\langle\nabla \eta ,\mathfrak{w}\rangle+\langle\widehat{\mathbf{g}},\mathfrak{w}\rangle+\langle\U,\mathfrak{w}\rangle  \nonumber\\&\leq (\eta,\mathrm{div\ }\mathfrak{w})_{\mathrm{L}^2}+\langle\widehat{\mathbf{g}},\mathfrak{w}\rangle +\langle\U,\mathfrak{w}\rangle\nonumber\\&\leq \|\eta\|_{\mathrm{L}^2}\|\mathrm{div\ }\mathfrak{w}\|_{\mathrm{L}^2}+\|\widehat{\mathbf{g}}\|_{\mathbb{H}^{-1}}\|\nabla\mathfrak{w}\|_{\mathbb{L}^2}+\|\U\|_{\H^{-1}}\|\nabla\mathfrak{w}\|_{\mathbb{L}^2}\nonumber\\&\leq \sqrt{2} \|\eta\|_{\mathrm{L}^2}\|\nabla\mathfrak{w}\|_{\L^2}+\frac{\alpha}{4}\|\nabla\mathfrak{w}\|_{\mathbb{L}^2}^2+\frac{2}{\alpha}\|\widehat{\mathbf{g}}\|_{\mathbb{H}^{-1}}^2+\frac{2}{\alpha}\|\U\|_{\mathbb{H}^{-1}}^2\nonumber\\&\leq \frac{3\alpha}{8}\|\nabla\mathfrak{w}\|_{\L^2}^2+\frac{4}{\alpha}\|\eta\|_{\mathrm{L}^2}^2+\frac{2}{\alpha}\|\widehat{\mathbf{g}}\|_{\mathbb{H}^{-1}}^2+\frac{2}{\alpha}\|\U\|_{\mathbb{H}^{-1}}^2.
\end{align}
Let us now take inner product with $\eta(\cdot)$ to the second equation in \eqref{4} to get 
\begin{align}\label{217}
\frac{1}{2}\frac{\d}{\d t}\|\eta(t)\|_{\mathrm{L}^2}^2&=-(\mathrm{div}(h\mathfrak{w}(t)),\eta(t))_{\mathrm{L}^2}\leq \|\mathrm{div}(h\mathfrak{w}(t))\|_{\mathrm{L}^2}\|\eta(t)\|_{\mathrm{L}^2}\nonumber\\&=\|h\mathrm{div\ }\mathfrak{w}(t)+\nabla h\cdot\mathfrak{w}(t)\|_{\mathrm{L}^2}\|\eta(t)\|_{\mathrm{L}^2}\nonumber\\&\leq \left( \|h\mathrm{div\ }\mathfrak{w}(t)\|_{\mathrm{L}^2}+\|\nabla h\cdot\mathfrak{w}(t)\|_{\mathrm{L}^2}\right)\|\eta(t)\|_{\mathrm{L}^2}\nonumber\\&\leq \left(\|h\|_{\mathrm{L}^{\infty}}\|\mathrm{div\ }\mathfrak{w}(t)\|_{\mathrm{L}^2}+\|\nabla h\|_{\mathbb{L}^{\infty}}\|\mathfrak{w}(t)\|_{\mathbb{L}^2}\right)\|\eta(t)\|_{\mathrm{L}^2}\nonumber\\&\leq \sqrt{2}\mu\|\nabla\mathfrak{w}(t)\|_{\mathbb{L}^2}\|\eta(t)\|_{\mathrm{L}^2}+M\|\mathfrak{w}(t)\|_{\mathbb{L}^2}\|\eta(t)\|_{\mathrm{L}^2}\nonumber\\&\leq \frac{\alpha}{8}\|\nabla\mathfrak{w}(t)\|_{\mathbb{L}^2}^2+\frac{M}{2}\|\mathfrak{w}(t)\|_{\mathbb{L}^2}^2+\left(\frac{4\mu^2}{\alpha}+\frac{M}{2}\right)\|\eta(t)\|_{\mathrm{L}^2}^2,
\end{align}
where we used H\"older's and Young's inequalities. Combining \eqref{216} and \eqref{217}, we find 
\begin{align*}
&\frac{\d }{\d t}\left(\|\mathfrak{w}(t)\|_{\mathbb{L}^2}^2+\|\eta(t)\|_{\mathrm{L}^2}^2\right)+\alpha\|\nabla\mathfrak{w}(t)\|_{\mathbb{L}^2}^2\nonumber\\&\leq -2(\B'(\widehat{\mathbf{u}}(t))\mathfrak{w}(t),\mathfrak{w}(t))_{\mathbb{L}^2}+ M\|\mathfrak{w}(t)\|_{\mathbb{L}^2}^2+\left(\frac{8(\mu^2+1)}{\alpha}+M\right)\|\eta(t)\|_{\mathrm{L}^2}^2\nonumber\\&\quad+\frac{4}{\alpha}\|\widehat{\mathbf{g}}(t)\|_{\mathbb{H}^{-1}}^2+\frac{4}{\alpha}\|\U(t)\|_{\mathbb{H}^{-1}}^2.
\end{align*}
Integrating the above inequality from $0$ to $t$, we obtain 
\begin{align}\label{218}
&\|\mathfrak{w}(t)\|_{\mathbb{L}^2}^2+\|\eta(t)\|_{\mathrm{L}^2}^2+\alpha\int_0^t\|\nabla\mathfrak{w}(s)\|_{\mathbb{L}^2}^2\d s+4\gamma\int_0^t\||\u(s)+\w^0(s)|^{\frac{1}{2}}\mathfrak{w}(s)\|_{\L^2}^2\d s\nonumber\\&\leq \|\mathfrak{w}_0\|_{\mathbb{L}^2}^2+\|\eta_0\|_{\mathrm{L}^2}^2+\left(\frac{8(\mu^2+1)}{\alpha}+M\right)\int_0^t\left(\|\mathfrak{w}(s)\|_{\mathbb{L}^2}^2+\|\eta(s)\|_{\mathrm{L}^2}^2\right)\d s\nonumber\\&\quad +\frac{4}{\alpha}\int_0^t\|\widehat{\mathbf{g}}(s)\|_{\mathbb{H}^{-1}}^2\d s+\frac{4}{\alpha}\int_0^t\|\U(s)\|_{\mathbb{H}^{-1}}^2\d s,
\end{align}
where we used the fact that 
\begin{align*}
{\int_0^t(\B'(\widehat{\mathbf{u}}(t))\mathfrak{w}(t),\mathfrak{w}(t))_{\mathbb{L}^2}\d t=2\int_0^t\int_{\Omega}\gamma(x)|\widehat{\u}(t,x)+\w^0(t,x)||\mathfrak{w}(t,x)|^2\d x\d t\geq 0,} 
\end{align*}
for all $t\in[0,T]$. An application of Gornwall's inequality in \eqref{218} yields 
\begin{align}\label{219}
&\|\mathfrak{w}(t)\|_{\mathbb{L}^2}^2+\|\eta(t)\|_{\mathrm{L}^2}^2\nonumber\\&\leq \left( \|\mathfrak{w}_0\|_{\mathbb{L}^2}^2+\|\eta_0\|_{\mathrm{L}^2}^2+\frac{4}{\alpha}\int_0^T\|\widehat{\mathbf{g}}(t)\|_{\mathbb{H}^{-1}}^2\d t+\frac{4}{\alpha}\int_0^T\|\U(t)\|_{\mathbb{H}^{-1}}^2\d t\right)e^{\left(\frac{8(\mu^2+1)}{\alpha}+M\right)T},
\end{align}
for all $t\in[0,T]$. Thus, from \eqref{218}, it is immediate that 
\begin{align*}
&\sup_{t\in[0,T]}\left(\|\mathfrak{w}(t)\|_{\mathbb{L}^2}^2+\|\eta(t)\|_{\mathrm{L}^2}^2\right)+\alpha\int_0^T\|\nabla\mathfrak{w}(t)\|_{\mathbb{L}^2}^2\d t\nonumber\\&\leq  \left( \|\mathfrak{w}_0\|_{\mathbb{L}^2}^2+\|\eta_0\|_{\mathrm{L}^2}^2+\frac{4}{\alpha}\int_0^T\|\widehat{\mathbf{g}}(t)\|_{\mathbb{H}^{-1}}^2\d t+\frac{4}{\alpha}\int_0^T\|\U(t)\|_{\mathbb{H}^{-1}}^2\d t\right)e^{2\left(\frac{8(\mu^2+1)}{\alpha}+M\right)T}. 
\end{align*}

Let us now obtain the  estimates of time derivatives. For any $\v\in\mathrm{L}^2(0,T;\H_0^1(\Omega))$, we have 
\begin{align}\label{213}
&\left|\int_0^T\langle\partial_t\mathfrak{w}(t),\v(t)\rangle\d t\right|\nonumber\\&\leq\int_0^T|\langle\A\mathfrak{w}(t),\v(t)\rangle|\d t+\int_0^T|\langle\B'(\widehat{\mathbf{u}}(t))\mathfrak{w}(t),\v(t)\rangle|\d t+\int_0^T|\langle \nabla \eta (t),\v(t)\rangle|\d t\nonumber\\&\quad+\int_0^T|\langle\widehat{\mathbf{g}}(t),\v(t)\rangle|\d t+\int_0^T|\langle\U(t),\v(t)\rangle|\d t\nonumber\\&\leq\alpha\int_0^T\|\nabla\mathfrak{w}(t)\|_{\L^2}\|\nabla\v(t)\|_{\L^2}\d t+\beta\int_0^T\|\mathfrak{w}(t)\|_{\L^2}\|\v(t)\|_{\L^2}\d t\nonumber\\&\quad+\frac{2C_{\Omega}r}{\lambda}\int_0^T\left(\|\u(t)\|_{\L^4}+\|\w^0(t)\|_{\L^4}\right)\|\mathfrak{w}(t)\|_{\L^4}\|\nabla\v(t)\|_{\L^2}\d t\nonumber\\&\quad+\int_0^T\|\eta(t)\|_{\mathrm{L}^2}\|\mathrm{div\ }\v(t)\|_{\mathrm{L}^2}\d t+\int_0^T\|\widehat{\g}(t)\|_{\H^{-1}}\|\nabla\v(t)\|_{\L^2}\d t\nonumber\\&\quad+\int_0^T\|\U(t)\|_{\H^{-1}}\|\nabla\v(t)\|_{\L^2}\d t\nonumber\\&\leq\Bigg\{\alpha\left(\int_0^T\|\nabla\mathfrak{w}(t)\|_{\L^2}^2\d t\right)^{1/2}+\sqrt{2}\left(\int_0^T\|\eta(t)\|_{\mathrm{L}^2}^2\d t\right)^{1/2}\nonumber\\&\quad+\beta C_{\Omega}\left(\int_0^T\|\mathfrak{w}(t)\|_{\L^2}^2\d t\right)^{1/2}\nonumber\\&\quad+\frac{4C_{\Omega}r}{\lambda}\left(\int_0^T\|\u(t)\|_{\L^4}^4\d t+\int_0^T\|\w^0(t)\|_{\L^4}^4\d t\right)^{1/4}\left(\int_0^T\|\mathfrak{w}(t)\|_{\L^4}^4\d t\right)^{1/4}\nonumber\\&\quad+\left(\int_0^T\|\widehat{\g}(t)\|_{\H^{-1}}^2\d t\right)^{1/2}+\left(\int_0^T\|\U(t)\|_{\H^{-1}}^2\d t\right)^{1/2} \Bigg\}\left(\int_0^T\|\nabla\v(t)\|_{\L^2}\d t\right)^{1/2},
\end{align}
where we used \eqref{8}, \eqref{29}, Lemma \ref{lem2.3} (vii), H\"older's and Young's inequalities. Thus, we obtain  $\|\partial_t\mathfrak{w}\|_{\mathrm{L}^2(0,T;\H^{-1}(\Omega))}<+\infty$. Similarly, for every $\upsilon\in\mathrm{L}^2(0,T;\mathrm{L}^2(\Omega))$, we have 
\begin{align}\label{214}
&\left|\int_0^T(\partial_t\eta(t),\upsilon(t))_{\mathrm{L}^2}\d t\right|\nonumber\\&\leq  \int_0^T|(\mathrm{div}(h\mathfrak{w}(t)),\upsilon(t))_{\mathrm{L^2}}|\d t\\&\leq \bigg[M\left(\int_0^T\|\mathfrak{w}(t)\|_{\L^2}^2\d t\right)^{1/2}+\sqrt{2}\mu\left(\int_0^T\|\nabla\mathfrak{w}(t)\|_{\L^2}^2\d t\right)^{1/2}\bigg]\left(\int_0^T\|\upsilon(t)\|_{\mathrm{L}^2}^2\d t\right)^{1/2}, \nonumber
\end{align}
so that $\|\partial_t\eta\|_{\mathrm{L}^2(0,T;\mathrm{L}^2(\Omega))}<+\infty$. Moreover, using a standard Faedo-Galerkin approximation technique and the Banach-Alaoglu theorem (which ensures the weakly convergent subsequences),  we have the following existence and uniqueness Theorem. As the system is linear, the estimate \eqref{219} easily implies the uniqueness of weak solution. 
\begin{theorem}\label{linear}
	Let  $(\mathfrak{w}_0,\eta_0)\in\mathbb{L}^2(\Omega)\times\mathrm{L}^2(\Omega)$ be given.  For $\widehat{\g}\in\mathrm{L}^2(0,T;\mathbb{H}^{-1}(\Omega))$ and $\U\in\mathrm{L}^2(0,T;\mathbb{H}^{-1}(\Omega))$,
	there exists \emph{a unique weak solution} to the system \eqref{4} satisfying 
	$$(\mathfrak{w},\eta)\in(\mathrm{C}([0,T];\mathbb{L}^2(\Omega))\cap\mathrm{L}^2(0,T;\mathbb{H}_0^1(\Omega)))\times\mathrm{C}([0,T];\mathrm{L}^2(\Omega)),$$ with $$(\partial_t\mathfrak{w},\partial_t \eta)\in\mathrm{L}^2(0,T;\mathbb{H}^{-1}(\Omega))\times\mathrm{L}^2(0,T;\mathrm{L}^2(\Omega)).$$
\end{theorem}

\section{Optimal Control Problem}\label{sec3}\setcounter{equation}{0}  In this section, we formulate a distributed optimal control problem  as the minimization of a suitable cost functional subject to the controlled tidal dynamics system \eqref{1}. The main objective  is to prove the existence of an optimal control that minimizes the cost functional given below (see \eqref{cost}), subject to the constraint  
\begin{equation}\label{a1}
\left\{
\begin{aligned}
\frac{\partial \mathbf{u}(t)}{\partial t}+\A\mathbf{u}(t)+\B(\mathbf{u}(t))+\nabla \xi (t)&=\mathbf{f}(t)+\U(t),\ \text{ in }\ \H^{-1}(\Omega),\\
\frac{\partial \xi(t)}{\partial t}+\mathrm{div}(h \mathbf{u}(t))&=0, \ \text{ in }\ \mathrm{L}^2(\Omega),\\
\mathbf{u}(0)&=\mathbf{u}_0,\ \xi(0)=\xi_0,
\end{aligned}
\right.
\end{equation}
for a.e. $t\in[0,T]$  and establish the first order necessary conditions of optimality.  The cost functional under our consideration is given  by
\begin{eqnarray}\label{cost}
\begin{aligned}
&\mathscr{J}(\u,\xi,\U) \\&:= \frac{1}{2} \int_0^T \|\u(t)-\u_d(t)\|^2_{\mathbb{L}^2} \d t+\frac{1}{2}\int_0^T\|\nabla\u(t)\|_{\L^2}^2\d t+ \frac{1}{2} \int_0^T \|\xi(t) -\xi_d(t)\|^2_{\mathrm{L}^2} \d t\\&\quad+ \frac{1}{2} \int_0^T\|\U(t)\|^2_{\mathbb{H}^{-1}}\d t+\frac{1}{2}\|\u(T)-\u^f\|^2_{\L^2}+\frac{1}{2}\|\xi(T)-\xi^f\|^2_{\mathrm{L}^2},
\end{aligned}
\end{eqnarray}
where $\u_d(\cdot)\in\mathrm{L}^2(0,T;\mathbb{L}^2(\Omega))$ and $\varphi_d(\cdot)\in\mathrm{L}^2(0,T;\mathrm{L}^2(\Omega))$ are the desired velocity field (or reference velocity) and desired elevation (or tidal flow), respectively. Furthermore, $\u^f\in\L^2(\Omega)$ represents the desired velocity at time $T$ and $\xi^f\in\mathrm{L}^2(\Omega)$ denotes the desired elevation at time $T$. Note that the cost functional is the sum of the total energy,  total dissipation of energy of the flow and total effort by controls. Physically, one can think it as an optimal estimation problem, where we are trying to find an unknown external force based on measurements and the cost functional is then the difference between measurement and the tidal dynamics.  In this work, we take the set of all admissible control class $\mathscr{U}_{\mathrm{ad}}$ as $\mathrm{L}^{2}(0,T;\mathbb{H}^{-1}(\Omega))$. One can consider $\mathrm{L}^{2}(0,T;\mathbb{L}^{2}(\Omega))$ also as a admissible control class (in that case, we have to replace $ \frac{1}{2} \int_0^T\|\U(t)\|^2_{\mathbb{H}^{-1}}\d t$ with $ \frac{1}{2} \int_0^T\|\U(t)\|^2_{\mathbb{L}^{2}}\d t$ in the cost functional \eqref{cost}), but the control class $\mathscr{U}_{\mathrm{ad}}$, under our consideration is much larger than $\mathrm{L}^{2}(0,T;\mathbb{L}^{2}(\Omega))$. As the existence and uniqueness of weak solution is known for the   basic state $(\u,\xi),$ for any control $\U\in\mathrm{L}^{2}(0,T;\mathbb{H}^{-1}(\Omega))$, and we are working only with  the weak solution regularity class,  we are able to take such an admissible class.  Next, we give the definition of admissible class of solutions.

\begin{definition}[Admissible class]\label{definition 1}
	The \emph{admissible class} $\mathscr{A}_{\mathrm{ad}}$ of triples $$(\u,\xi,\U)\in(\mathrm{C}([0,T];\mathbb{L}^2(\Omega))\cap\mathrm{L}^2(0,T;\mathbb{H}_0^1(\Omega)))\times\mathrm{C}([0,T];\mathrm{L}^2(\Omega))\times\mathscr{U}_{\mathrm{ad}} $$ is defined as the set of states $(\u,\xi)$ solving the system \eqref{a1} with the control $\U \in \mathscr{U}_{ad}$. That is,
	\begin{align*}
	\mathscr{A}_{\mathrm{ad}}&:=\Big\{(\u,\xi,\U) :(\u,\xi)\in(\mathrm{C}([0,T];\mathbb{L}^2(\Omega))\cap\mathrm{L}^2(0,T;\mathbb{H}_0^1(\Omega)))\times\mathrm{C}([0,T];\mathrm{L}^2(\Omega))\nonumber\\&\qquad \text{ is \text{the unique weak solution} of }\eqref{a1}  \text{ with the control }\U\in\mathscr{U}_{\mathrm{ad}}\Big\}.
	\end{align*}
\end{definition}
Clearly $\mathscr{A}_{\mathrm{ad}}$ is a nonempty set as for any $\U \in \mathscr{U}_{\mathrm{ad}}$, there exists \emph{a unique weak solution} of the system \eqref{a1}. In view of the above definition, the optimal control problem we are considering can be  formulated as:
\begin{align}\label{control problem}
\min_{ (\u,\xi,\U) \in \mathscr{A}_{\mathrm{ad}}}  \mathscr{J}(\u,\xi,\U).
\end{align}
A solution to the problem \eqref{control problem} is called an \emph{optimal solution}. The optimal triplet is denoted by $(\u^* ,\xi^*, \U^*)$ and the control $\U^*$ is called an \emph{optimal control}.

\subsection{Existence of an optimal control}
Our first aim is to show that an optimal triplet $(\u^*,\xi^*,\U^*)$ exists for the problem \eqref{control problem}.
\begin{theorem}[Existence of an optimal triplet]\label{optimal}
	Let  $(\mathbf{u}_0,\xi_0)\in\mathbb{L}^2(\Omega)\times\mathrm{L}^2(\Omega)$ and  $\f\in\mathrm{L}^2(0,T;\mathbb{H}^{-1}(\Omega))$ be given.  Then there exists at least one triplet  $(\u^*,\xi^*,\U^*)\in\mathscr{A}_{\mathrm{ad}}$  such that the functional $ \mathscr{J}(\u,\xi,\U)$ attains its minimum at $(\u^*,\xi^*,\U^*)$, where $(\u^*,\xi^*)$ is the unique weak solution of the system  \eqref{a1}  with the control $\U^*$.
\end{theorem}
\begin{proof}
	Let us first define $$\mathscr{J} := \inf \limits _{\U \in \mathscr{U}_{\mathrm{ad}}}\mathscr{J}(\u,\xi,\U).$$
	Since, $0\leq \mathscr{J} < +\infty$, there exists a minimizing sequence $\{\U_n\} \in \mathscr{U}_{\mathrm{ad}}$ such that $$\lim_{n\to\infty}\mathscr{J}(\u_n,\xi_n,\U_n) = \mathscr{J},$$ where $(\u_n,\xi_n)$ is the unique weak solution of the system:
	\begin{equation}\label{313}
	\left\{
	\begin{aligned}
	\frac{\partial \mathbf{u}_n(t)}{\partial t}+\A\mathbf{u}_n(t)+\B(\mathbf{u}_n(t))+\nabla \xi_n (t)&=\mathbf{f}(t)+\U_n(t),\ \text{ in }\ \H^{-1}(\Omega),\\
	\frac{\partial \xi_n(t)}{\partial t}+\mathrm{div}(h \mathbf{u}_n(t))&=0, \ \text{ in }\ \mathrm{L}^2(\Omega),\\
	(	\mathbf{u}_n(0),\xi_n(0))&=(\mathbf{u}_0,\xi_0) \ \text{ in }\ \L^2(\Omega)\times\mathrm{L}^2(\Omega). 
	\end{aligned}
	\right.
	\end{equation}
	Since  $\mathbf{0}\in\mathscr{U}_{\mathrm{ad}}$, without loss of generality, we may assume that $\mathscr{J}(\u_n,\xi_n,\U_n) \leq \mathscr{J}(\u,\xi,\mathbf{0})$, where $(\u,\xi,\mathbf{0})\in\mathscr{A}_{\mathrm{ad}}$. Using the definition of $\mathscr{J}(\cdot,\cdot,\cdot)$, this immediately gives
	\begin{align}\label{bound}
	&\frac{1}{2} \int_0^T \|\u_n(t)-\u_d(t)\|^2_{\mathbb{L}^2} \d t+\frac{1}{2}\int_0^T\|\nabla\u_n(t)\|_{\L^2}^2\d t+ \frac{1}{2} \int_0^T \|\xi_n(t) -\xi_d(t)\|^2_{\mathrm{L}^2} \d t\nonumber\\&\quad+ \frac{1}{2} \int_0^T\|\U_n(t)\|^2_{\mathbb{H}^{-1}}\d t+\frac{1}{2}\|\u_n(T)-\u^f\|^2_{\L^2}+\frac{1}{2}\|\xi_n(T)-\xi^f\|^2_{\mathrm{L}^2},\nonumber\\& \leq \frac{1}{2} \int_0^T \|\u(t)-\u_d(t)\|^2_{\mathbb{L}^2} \d t+\frac{1}{2}\int_0^T\|\nabla\u(t)\|_{\L^2}^2\d t+ \frac{1}{2} \int_0^T \|\xi(t) -\xi_d(t)\|^2_{\mathrm{L}^2} \d t\nonumber\\&\quad+\frac{1}{2}\|\u(T)-\u^f\|^2_{\L^2}+\frac{1}{2}\|\xi(T)-\xi^f\|^2_{\mathrm{L}^2}\nonumber\\&\leq \left(\|\u_0\|_{\mathbb{L}^2}^2+\|\xi_0\|_{\mathrm{L}^2}^2+\frac{r}{\lambda}\int_0^T\|\w^0(t)\|_{\mathbb{L}^4}^4\d t+\int_0^T\|\f(t)\|_{\H^{-1}}^2\d t\right)e^{KT}\nonumber\\&\quad+\int_0^T\|\u_d(t)\|_{\L^2}^2\d t+\int_0^T\|\xi_d(t)\|_{\mathrm{L}^2}^2\d t+\|\u^f\|_{\L^2}^2+\|\xi^f\|_{\mathrm{L}^2}^2, \end{align}
	where we used \eqref{energy}. 	Since $\u_d \in \mathrm{L}^{2}(0,T;\L^2(\Omega))$, $\xi_d \in \mathrm{L}^{2}(0,T;\mathrm{L}^2(\Omega))$, $\u^f\in\L^2(\Omega)$ and $\xi^f\in\mathrm{L}^2(\Omega)$, from the above relation, it is clear that, there exists an $R>0$, large enough such that
	$$0 \leq \mathscr{J}(\u_n,\xi_n,\U_n) \leq R < +\infty.$$
	In particular, there exists a large $\widetilde{C}>0,$ such that
	$$ \int_0^T \|\U_n(t)\|^2_{\H^{-1}} \d t \leq  \widetilde{C} < +\infty .$$
	That is, the sequence $\{\U_n\}$ is uniformly bounded in the space $\mathrm{L}^2(0,T;\H^{-1}(\Omega))$. Since $(\u_n,\xi_n)$ is the unique weak solution of the system \eqref{a1} with the control $\U_n$, from the energy estimate \eqref{energy}, we have 
	\begin{align*}
	&\|\u_n(t)\|_{\mathbb{L}^2}^2+\|\xi_n(t)\|_{\mathrm{L}^2}^2+\alpha\int_0^t\|\nabla\u_n(s)\|_{\mathbb{L}^2}^2\d s\nonumber\\&\leq  \left(\|\u_0\|_{\mathbb{L}^2}^2+\|\xi_0\|_{\mathrm{L}^2}^2+\frac{r}{\lambda}\int_0^T\|\w^0(t)\|_{\mathbb{L}^4}^4\d s+2\int_0^T\|\f(t)\|_{\H^{-1}}^2\d s\right.\nonumber\\&\qquad\left.+2\int_0^T\|\U_n(t)\|_{\H^{-1}}^2\d t\right)e^{KT}\nonumber\\&\leq C \left(\|\u_0\|_{\mathbb{L}^2}^2+\|\xi_0\|_{\mathrm{L}^2}^2+\frac{r}{\lambda}\int_0^T\|\w^0(t)\|_{\mathbb{L}^4}^4\d t+2\int_0^T\|\f(t)\|_{\H^{-1}}^2\d t+2\widetilde{C}\right)e^{KT},
	\end{align*}
	for all $t\in[0,T]$ (one can use \eqref{bound} also for uniform bound).	It can be easily seen  that the sequence $\{\u_n\} $ is uniformly bounded in $\mathrm{L}^{\infty}(0,T;\L^2(\Omega))\cap \mathrm{L}^2(0,T;\H^1_0(\Omega))$ and $\{\xi_n\} $ is uniformly bounded in $\mathrm{L}^{\infty}(0,T;\mathrm{L}^2(\Omega))$. 
	Hence, by using the Banach-Alaglou theorem, we can extract a subsequence $\{(\u_{n_k},\xi_{n_k}, \U_{n_k})\}$ such that 
	\begin{equation}\label{conv}
	\left\{
	\begin{aligned} 
	\u_{n_k}&\xrightharpoonup{w^*}\u^*\text{ in } \mathrm{L}^{\infty}(0,T;\L^2(\Omega)), \\ \u_{n_k} &\xrightharpoonup{w} \u^*\text{ in  }\textrm{L}^2(0,T;\H^1_0(\Omega)),\\ \xi_{n_k} &\xrightharpoonup{w^*} \xi^*\text{ in }\mathrm{ L}^{\infty}(0,T;\mathrm{L}^2(\Omega)),\\ 
	\U_{n_k} &\xrightharpoonup{w}  \U^*\text{ in  }\mathrm{L}^2(0,T;\H^{-1}(\Omega)),
	\end{aligned}
	\right.
	\end{equation}
	as $k\to\infty$. For any $\v\in\mathrm{L}^2(0,T;\H_0^1(\Omega))$, we have 
	\begin{align*}
	&\left|\int_0^T\langle\partial_t\u_n(t),\v(t)\rangle\d t\right|\nonumber\\&\leq\int_0^T|\langle\A\u_n(t),\v(t)\rangle|\d t+\int_0^T|\langle\B(\u_n(t)),\v(t)\rangle|\d t+\int_0^T|\langle\nabla\xi_n(t),\v(t)\rangle|\d t\nonumber\\&\quad+\int_0^T|\langle\f(t),\v(t)\rangle|\d t+\int_0^T|\langle\U(t),\v(t)\rangle|\d t\nonumber\\&\leq\bigg\{\alpha\left(\int_0^T\|\nabla\u_n(t)\|_{\L^2}^2\d t\right)^{1/2}+\beta\left(\int_0^T\|\u_n(t)\|_{\L^2}^2\d t\right)^{1/2}+\left(\int_0^T\|\xi_n(t)\|_{\mathrm{L}^2}^2\d t\right)^{1/2}\nonumber\\&\quad+\frac{2\sqrt{2}C_{\Omega}r}{\lambda}\left(\int_0^T\|\u_n(t)\|_{\L^4}^4\d t+\int_0^T\|\w^0(t)\|_{\L^4}^4\d t\right)^{1/2}+\left(\int_0^T\|\f(t)\|_{\H^{-1}}^2\d t\right)^{1/2}\nonumber\\&\quad+\left(\int_0^T\|\U(t)\|_{\H^{-1}}^2\d t\right)^{1/2}\bigg\}\left(\int_0^T\|\nabla\v(t)\|_{\L^2}^2\d t\right)^{1/2},
	\end{align*}
	so that $\|\partial_t\u_n\|_{\mathrm{L}^2(0,T;\H^{-1}(\Omega))}$ is uniformly bounded. Similar to the above calculation, for any  $\upsilon\in\mathrm{L}^2(0,T;\mathrm{L}^2(\Omega))$, we obtain 
	\begin{align}
	&	\left|\int_0^T(\partial_t\xi_n(t),\upsilon(t))_{\mathrm{L}^2}\d t\right| \nonumber\\&\leq\int_0^T|(\mathrm{div}(h\u_n(t)),\upsilon(t))_{\mathrm{L}^2}|\d t\nonumber\\&\leq \Bigg[M\left(\int_0^T\|\u_n(t)\|_{\L^2}^2\d t\right)^{1/2}+\sqrt{2}\mu\left(\int_0^T\|\nabla\u_n(t)\|_{\L^2}^2\d t\right)^{1/2}\Bigg]\nonumber\\&\quad\times\left(\int_0^T\|\upsilon(t)\|_{\mathrm{L}^2}^2\d t\right)^{1/2}, 
	\end{align} 
	so that $\|\partial_t\xi_n\|_{\mathrm{L}^2(0,T;\mathrm{L}^2(\Omega))}$ is uniformly bounded. Thus, along a subsequence, we get the following convergence also:
	\begin{equation}\label{dev}
	\left\{
	\begin{aligned}
	\partial_t	\u_{n_k}&\xrightharpoonup{w}\partial_t\u^*\text{ in } \mathrm{L}^{2}(0,T;\H^{-1}(\Omega)),\\
	\partial_t \xi_{n_k} &\xrightharpoonup{w}\partial_t \xi^*\text{ in }\mathrm{ L}^{2}(0,T;\mathrm{L}^2(\Omega)),
	\end{aligned}
	\right. 	
	\end{equation}
	as $k\to\infty$.	Using Aubin-Lion's compactness theorem (see Theorem 1, \cite{SJ})  and the convergence given in \eqref{conv}, we infer that 
	\begin{equation}\label{strong}
	\begin{aligned}
	\u_{n_k}&\to \u^*\ \text{ in }\ \mathrm{L}^{2}(0,T;\L^2(\Omega)), 
	\end{aligned}
	\end{equation}
	as $k\to\infty$. Using H\"older's and Ladyzhenskaya's  inequalities, we find 
	\begin{align}\label{L4}
	&\int_0^T\|\u_n(t)-\u(t)\|_{\L^4}^2\d t\nonumber\\&\leq C_{\Omega}\sup_{t\in[0,T]}\|\u_n(t)-\u(t)\|_{\L^2}\int_0^T\|\nabla(\u_n(t)-\u(t))\|_{\L^2}\d t\nonumber\\&\leq \sqrt{2}C_{\Omega}T^{1/2}\sup_{t\in[0,T]}\|\u_n(t)-\u(t)\|_{\L^2}\nonumber\\&\quad\times\left[\left(\int_0^T\|\nabla\u_n(t)\|_{\L^2}^2\d t\right)^{1/2}+\left(\int_0^T\|\nabla\u(t)\|_{\L^2}^2\d t\right)^{1/2}\right]\nonumber\\&\to 0\ \text{ as } \ n\to\infty, 
	\end{align}
	using \eqref{conv} and \eqref{strong}. 
	Most of the terms appearing in the first and second equations in \eqref{313} are linear, we can use the weak convergences in \eqref{conv} and \eqref{dev} to pass the limit in the weak form (see \eqref{3.13} for weak formulation of the uncontrolled system). Thus, we need to check  the convergence of the nonlinear term only. For any $\v\in\mathrm{C}^1(\overline{\mathrm{Q}}_T)$, we have 
	\begin{align*}
	&\left|\int_0^T(\B(\u_n(t))-\B(\u(t)),\v(t))_{\mathbb{L}^2}\d t\right|\nonumber\\&\leq\int_0^T\|\B(\u_n(t))-\B(\u(t))\|_{\L^2}\|\v(t)\|_{\L^2}\d t\nonumber\\&\leq\frac{r}{2\lambda}\sup_{t\in[0,T]}\|\v(t)\|_{\L^2}\left(\int_0^T\left[\|\u_n(t)\|_{\L^4}+\|\u(t)\|_{\L^4}+\|\w^0(t)\|_{\L^4}\right]^2\d t\right)^{1/2}\nonumber\\&\quad\times\left(\int_0^T\|\u_n(t)-\u(t)\|_{\L^4}^2\d t\right)^{1/2}\nonumber\\&\to 0\ \text{ as } \ n\to\infty, 
	\end{align*}
	where we used  the Lemma \ref{lem2.3} (v) and \eqref{L4}. Using a density argument, we obtain 
	\begin{align}
	\B(\u_n)\xrightarrow{w}\B(\u)\ \text{ in }\ \mathrm{L}^2(0,T;\L^2(\Omega)), \ \text{ as } \ n\to\infty. 
	\end{align}
	From the above convergences and discussions, one can pass limit in the equation corresponding to $(\u_{n_k},\xi_{n_k},\U_{n_k})\in \mathscr{A}_{\mathrm{ad}}$ (replace $n$ with $n_k$ in \eqref{313}) and conclude that  $(\u^*,\xi^*)$ is the unique weak solution of the system \eqref{a1} with the control $\U^*\in\mathrm{L}^2(0,T;\H^{-1}(\Omega))$.   It is a consequence of Theorems 2 and 3, section 5.9.2, \cite{evans} that $(\u^*,\xi^*) \in \mathrm{C}([0,T];\L^2(\Omega)) \times \mathrm{C}([0,T];\mathrm{L}^2(\Omega))$. Note that the initial condition given  in (\ref{313}) and the right continuity in time at $0$ gives 
	\begin{align}\label{initial2}
	\u^*(0)=	\u_0\in\mathbb{L}^2(\Omega)\ \text{ and }\ \xi^*(0)=\xi_0\in \mathrm{L}^2(\Omega).
	\end{align}
	Since, $(\u^*,\xi^*)$ is the unique weak solution of the system \eqref{a1} with the control $\U^*\in\mathrm{L}^2(0,T;\H^{-1}(\Omega))$, the whole sequence $(\u_n,\xi_n)$ converges to $(\u^*,\xi^*)$. This easily gives  $(\u^*,\xi^*,\U^*)\in \mathscr{A}_{\mathrm{ad}}$. 
	
	Now we show that  $(\u^*,\xi^*,\U^*)$ is a minimizer, that is, \emph{$\mathscr{J}=\mathscr{J}(\u^*,\xi^*,\U^*)$}.  Since the cost functional $\mathscr{J}(\cdot,\cdot,\cdot)$ is continuous and convex (see Proposition III.1.6 and III.1.10,  \cite{EI})  on $\mathrm{L}^2(0,T;\L^2(\Omega)) \times \mathrm{L}^2(0,T;\mathrm{L}^2(\Omega)) \times \mathscr{U}_{\mathrm{ad}}$, it follows that $\mathscr{J}(\cdot,\cdot,\cdot)$ is weakly lower semi-continuous (Proposition II.4.5, \cite{EI}). That is, for a sequence 
	$$(\u_n,\xi_n,\U_n)\xrightharpoonup{w}(\u^*,\xi^*,\U^*)\ \text{ in }\ \mathrm{L}^2(0,T;\L^2(\Omega)) \times \mathrm{L}^2(0,T;\mathrm{L}^2(\Omega)) \times \mathrm{L}^2(0,T;\H^{-1}(\Omega)),$$
	we have 
	\begin{align*}
	\mathscr{J}(\u^*,\xi^*,\U^*) \leq  \liminf \limits _{n\rightarrow \infty} \mathscr{J}(\u_n,\xi_n,\U_n).
	\end{align*}
	Therefore, we get
	\begin{align*}\mathscr{J} \leq \mathscr{J}(\u^*,\xi^*,\U^*) \leq  \liminf \limits _{n\rightarrow \infty} \mathscr{J}(\u_n,\xi_n,\U_n)=  \lim \limits _{n\rightarrow \infty} \mathscr{J}(\u_n,\xi_n,\U_n) = \mathscr{J},\end{align*}
	and hence $(\u^*,\xi^*,\U^*)$ is an optimizer of the problem \eqref{control problem}.
\end{proof}

\subsection{The adjoint system} In order to establish the first order necessary conditions of optimality, we need to find the adjoint system corresponding to \eqref{a1}. Remember that the optimal control obtained in the Theorem \ref{optimal} can be characterized via adjoint variable.  In this subsection, we formally derive the adjoint system corresponding to the problem \eqref{a1}.  Let us first define
\begin{equation}\label{n1n2}
\left\{
\begin{aligned} 
\mathcal{N}_1(\u,\xi, \U) &:= -\A\mathbf{u}-\B(\mathbf{u})-\nabla \xi +\U,  \\
\mathcal{N}_2(\u,\xi) &:= -\mathrm{div}(h\u).
\end{aligned}
\right.
\end{equation}
Then the tidal dynamics system \eqref{a1} can be rewritten as
$$\{\partial_t \u,\partial_t \xi\} = \{\mathcal{N}_1(\u,\xi, \U), \mathcal{N}_2(\u,\xi)\}.$$
We define the \emph{augmented cost functional} $\widetilde{\mathscr{J}}$ by
\begin{align*}&\widetilde{\mathscr{J}}(\u,\xi,\U,\p,\varphi)\nonumber\\& :=  \int_0^T \langle \p,\partial_t \u - \mathcal{N}_1(\u,\xi, \U) \rangle \d t+ \int_0^T \langle \varphi, \partial_t \xi - \mathcal{N}_2(\u,\xi) \rangle \d t - \mathscr{J}(\u,\xi,\U) , \end{align*}
where $\p$ and $\varphi$ denote the adjoint variables corresponding to $\u$ and $\xi$ respectively. 
Next, we derive the adjoint equations formally by differentiating the augmented cost functional $\widetilde{\mathscr{J}}$ in the Gateaux  sense with respect to each of its variables. The adjoint variables $ \p, \xi $ and $\U$ satisfy the following system:
\begin{equation}\label{e2}
\left\{
\begin{aligned}
-	\frac{\partial\p}{\partial t}  - [\partial_\u \mathcal{N}_1]^*\p - [\partial_\u \mathcal{N}_2]^* \varphi  &= \mathscr{J}_\u,  \\ 
-	\frac{\partial \varphi}{\partial t} - [\partial_{\xi} \mathcal{N}_1]^*\p - [\partial_{\xi} \mathcal{N}_2]^* \varphi &= \mathscr{J}_{\xi},  \\
- [\partial_\U \mathcal{N}_1]^*\p - [\partial_\U  \mathcal{N}_2]^* \varphi &= \mathscr{J}_\U, \\ \p\big|_{\partial\Omega}&=\mathbf{0},\\ \p(T,\cdot)=\mathbf{0}, \ \varphi(T,\cdot)&=0.
\end{aligned}
\right.
\end{equation}
Note that differentiating $\widetilde{\mathscr{J}}$ with respect to the adjoint variables recovers the original nonlinear system.
Further, we compute $[\partial_\u \mathcal{N}_1]^*\p$, $[\partial_\u \mathcal{N}_2]^* \varphi$, $[\partial_{\xi} \mathcal{N}_1]^*\p$, $[\partial_{\xi}  \mathcal{N}_2]^* \varphi$ as
\begin{equation}
\left\{
\begin{aligned}
{[\partial_\u \mathcal{N}_1]}^*\p &= -\widetilde{\A}\p-\B'(\u)\p,\ \quad  [\partial_\u \mathcal{N}_2]^* \varphi = h\nabla \varphi,\\
[\partial_\xi \mathcal{N}_1]^* \p &=  \mathrm{div\ }\p ,\qquad\quad\quad\quad \ \ [\partial_\xi \mathcal{N}_2]^* \varphi = 0,
\end{aligned}
\right.
\end{equation}
where $\widetilde{\A}\mathbf{p}=-\alpha \Delta \mathbf{p}-\beta \mathbf{k}\times \mathbf{p}$.  Since $\A$ is non-symmetric, we have $\A\v\neq \widetilde{\A}\v$. But one can easily see that $ \langle\A\v,\v\rangle=\alpha\|\nabla\v\|_{\L^2}^2=\langle \widetilde{\A}\mathbf{v},\v\rangle$, for all $\v\in\H_0^1(\Omega)$. Since $\mathbb{H}^{-1}(\Omega)$ is a separable, reflexive Banach space, it should be noted that $f(\cdot)=\frac{1}{2}\|\cdot\|^2_{\H^{-1}}$ is Gateaux differentiable. 
Note that the third condition in (\ref{e2}) gives $(-\Delta)^{-1}\U=-\p$. Thus from \eqref{e2}, it follows that the adjoint variables $(\p,\varphi)$ satisfy the following adjoint system in the abstract form: 
\begin{equation}\label{adj}
\left\{
\begin{aligned}
-\frac{\partial \p(t)}{\partial t}+\widetilde{\A}\p(t)+\B'(\u(t))\p (t)-h\nabla\varphi(t)&= (\u(t)-\u_d(t))-\Delta\u(t),  \text{ in } \H^{-1}(\Omega),\\
-\frac{\partial \varphi(t) }{\partial t}  -\mathrm{div\ }\p(t)&= (\xi(t) - \xi_d(t)),   \ \text{ in }\ \mathrm{L}^2(\Omega), \\  (\p(T),\varphi(T))=(\u(T)-\u^f,\xi(T)-\xi^f)&\in\L^2(\Omega)\times\mathrm{L}^2(\Omega),
\end{aligned}
\right.
\end{equation}
for a.e. $t\in[0,T]$. Note that $\p(x,t)=\mathbf{0}$, for a.e. $(x,t)\in\partial\Omega\times[0,T]$.  Let us now  obtain an a-priori energy estimate. We take inner product with $\p(\cdot)$ to the first equation in \eqref{adj} to obtain 
\begin{align}\label{37}
-&\frac{1}{2}\frac{\d}{\d t}\|\p(t)\|_{\mathbb{L}^2}^2+\alpha\|\nabla\p(t)\|_{\mathbb{L}^2}^2\nonumber\\&=-(\B'(\u(t))\p (t),\p(t))_{\mathbb{L}^2}-\langle h\nabla\varphi(t),\p(t)\rangle+(\u(t)-\u_d(t),\p(t))_{\L^2}\nonumber\\&\quad-\langle\Delta\u(t),\p(t)\rangle.
\end{align}
Using an  integration by parts, the Cauchy-Schwarz, H\"older and Young inequalities, we estimate the final three terms from the right hand side of the equality \eqref{37} as 
\begin{align}
-&\langle h\nabla\varphi,\p\rangle+(\u-\u_d,\p)_{\L^2}-\langle\Delta\u,\p\rangle\nonumber\\&=(\varphi,\mathrm{div}(h\p))_{\mathrm{L}^2}+(\u-\u_d,\p)_{\L^2}+(\nabla\u,\nabla\p)_{\mathbb{L}^2}\nonumber\\&\leq \|\varphi\|_{\mathrm{L}^2}\|\nabla h\cdot\p+h\mathrm{div\ }\p\|_{\mathrm{L}^2}+\|\u-\u_d\|_{\mathbb{L}^2}\|\p\|_{\mathbb{L}^2}+\|\nabla\u\|_{\L^2}\|\nabla\p\|_{\L^2}\nonumber\\&\leq \left(\|\nabla h\|_{\mathbb{L}^{\infty}}\|\p\|_{\mathbb{L}^2}+\|h\|_{\mathrm{L}^{\infty}}\|\mathrm{div\ }\p\|_{\mathrm{L}^2}\right)\|\varphi\|_{\mathrm{L}^2}+\frac{1}{2}\|\u-\u_d\|_{\mathbb{L}^2}^2+\frac{1}{2}\|\p\|_{\L^2}^2\nonumber\\&\quad+\frac{\alpha}{8}\|\nabla\p\|_{\L^2}^2+\frac{2}{\alpha}\|\nabla\u\|_{\L^2}^2\nonumber\\&\leq \frac{M+1}{2}\|\p\|_{\mathbb{L}^2}^2+\left(\frac{M}{2}+\frac{4\mu^2}{\alpha}\right)\|\varphi\|_{\mathrm{L}^2}^2+\frac{\alpha}{4}\|\nabla\p\|_{\L^2}^2+\frac{1}{2}\|\u-\u_d\|_{\mathbb{L}^2}^2+\frac{2}{\alpha}\|\nabla\u\|_{\L^2}^2.
\end{align}
Next, we take inner product with $\varphi(\cdot)$ to the second equation in \eqref{adj} to get 
\begin{align}\label{38}
-\frac{1}{2}\frac{\d}{\d t}\|\varphi(t)\|_{\mathrm{L}^2}^2&=(\mathrm{div\ }\p(t),\varphi(t))_{\mathrm{L}^2}+((\xi(t)-\xi_d(t)),\varphi(t))_{\mathrm{L}^2}\nonumber\\&\leq \|\mathrm{div\ }\p(t)\|_{\mathrm{L}^2}\|\varphi(t)\|_{\mathrm{L}^2}+\|\xi(t)-\xi_d(t)\|_{\mathrm{L}^2}\|\varphi(t)\|_{\mathrm{L}^2}\nonumber\\&\leq \frac{\alpha}{4}\|\nabla\p(t)\|_{\mathbb{L}^2}+2\left(\frac{1}{\alpha}+1\right)\|\varphi(t)\|_{\mathrm{L}^2}^2+\frac{1}{2}\|\xi(t)-\xi_d(t)\|_{\mathrm{L}^2}^2.
\end{align}
Combining \eqref{37} and \eqref{38}, we find
\begin{align*}
-&\frac{\d}{\d t}\left(\|\p(t)\|_{\mathbb{L}^2}^2+\|\varphi(t)\|_{\mathrm{L}^2}^2\right)+\alpha\|\nabla\p(t)\|_{\mathbb{L}^2}^2\nonumber\\&\leq -2(\B'(\u(t))\p (t),\p(t))_{\mathbb{L}^2}+ (M+1)\|\p(t)\|_{\mathbb{L}^2}^2\nonumber\\&\quad+\left[M+\frac{8\mu^2}{\alpha}+4\left(\frac{1}{\alpha}+1\right)\right]\|\varphi(t)\|_{\mathrm{L}^2}^2\nonumber\\&\quad+ \|\u(t)-\u_d(t)\|_{\mathbb{L}^2}^2+\|\xi(t)-\xi_d(t)\|_{\mathrm{L}^2}^2+\frac{4}{\alpha}\|\nabla\u(t)\|_{\L^2}^2.
\end{align*}
Integrating the above inequality from $t$ to $T$, we obtain 
\begin{align}\label{39}
&\|\p(t)\|_{\mathbb{L}^2}^2+\|\varphi(t)\|_{\mathrm{L}^2}^2+\alpha\int_t^T\|\nabla\p(s)\|_{\mathbb{L}^2}^2\d s+4\gamma\int_t^T\||\u(s)+\w^0(s)|^{\frac{1}{2}}\p(s)\|_{\L^2}^2\d s\nonumber\\&\leq\|\p(T)\|_{\mathbb{L}^2}^2+\|\varphi(T)\|_{\mathrm{L}^2}^2+ \left[M+\frac{8\mu^2}{\alpha}+4\left(\frac{1}{\alpha}+1\right)\right]\int_t^T\left(\|\p(t)\|_{\L^2}^2+\|\varphi(s)\|_{\mathrm{L}^2}^2\right)\d s\nonumber\\&\quad + \int_t^T\|\u(s)-\u_d(s)\|_{\mathbb{L}^2}^2\d s+\int_t^T\|\xi(s)-\xi_d(s)\|_{\mathrm{L}^2}^2\d s+\frac{4}{\alpha}\int_t^T\|\nabla\u(s)\|_{\L^2}^2\d s.
\end{align}
An application of the Gronwall's inequality in \eqref{39} yields 
\begin{align}\label{40}
&\|\p(t)\|_{\mathbb{L}^2}^2+\|\varphi(t)\|_{\mathrm{L}^2}^2\nonumber\\&\leq \bigg( \|\u(T)-\u^f\|_{\mathbb{L}^2}^2+\|\xi(T)-\xi^f\|_{\mathrm{L}^2}^2+\int_0^T\|\u(t)-\u_d(t)\|_{\mathbb{L}^2}^2\d t\nonumber\\&\quad+\int_0^T\|\xi(t)-\xi_d(t)\|_{\mathrm{L}^2}^2\d t+\frac{4}{\alpha}\int_0^T\|\nabla\u(t)\|_{\L^2}^2\d t\bigg)e^{\left[M+\frac{8\mu^2}{\alpha}+4\left(\frac{1}{\alpha}+1\right)\right]T},
\end{align}
for all $t\in[0,T]$. Using \eqref{40} in \eqref{39}, we find 
\begin{align}\label{40b}
&\sup_{t\in[0,T]}\left(\|\p(t)\|_{\mathbb{L}^2}^2+\|\varphi(t)\|_{\mathrm{L}^2}^2\right)+\alpha\int_0^T\|\nabla\p(t)\|_{\mathbb{L}^2}^2\d t\nonumber\\&\leq \bigg( \|\u(T)-\u^f\|_{\mathbb{L}^2}^2+\|\xi(T)-\xi^f\|_{\mathrm{L}^2}^2+\int_0^T\|\u(t)-\u_d(t)\|_{\mathbb{L}^2}^2\d t\nonumber\\&\quad+\int_0^T\|\xi(t)-\xi_d(t)\|_{\mathrm{L}^2}^2\d t+\frac{4}{\alpha}\int_0^T\|\nabla\u(t)\|_{\L^2}^2\d t\bigg)e^{2\left[M+\frac{8\mu^2}{\alpha}+4\left(\frac{1}{\alpha}+1\right)\right]T}.
\end{align}
Since the quantities $\u(T),\u^f\in\L^2(\Omega)$, $\xi(T), \xi^f\in\mathrm{L}^2(\Omega)$, $\u,\u_d\in\mathrm{L}^2(0,T;\mathbb{L}^2(\Omega))$, $\u\in\mathrm{L}^2(0,T;\H_0^1(\Omega))$ and $\varphi,\varphi_d\in\mathrm{L}^2(0,T;\mathrm{L}^2(\Omega))$, the right hand side of the estimate in \eqref{40b} is uniformly bounded. Calculations similar to \eqref{213} and \eqref{214} yield $\|\partial_t\p\|_{\mathrm{L}^2(0,T;\H^{-1}(\Omega))}$ and $\|\partial_t\varphi\|_{\mathrm{L}^2(0,T;\mathrm{L}^2(\Omega))}$ are also uniformly bounded. 

Once again using a Faedo-Galerkin approximation technique and the Banach-Alaoglu theorem, one can obtain the global solvability resuslts of \eqref{adj}. The following theorem gives the global  existence and uniqueness of weak solution to the system (\ref{adj}). As the system \eqref{adj} is linear, the uniqueness of weak solution easily follows from the estimate \eqref{40}. 

\begin{theorem}\label{thm3.4}
	For $\u_d\in\mathrm{L}^2(0,T;\mathbb{L}^2(\Omega))$,  $\xi_d\in\mathrm{L}^2(0,T;\mathrm{L}^2(\Omega))$, $\u^f\in\L^2(\Omega)$ and $\xi^f\in\mathrm{L}^2(\Omega)$, 
	there exists \emph{a unique weak solution} to the system \eqref{adj} satisfying 
	$$(\p,\varphi)\in(\mathrm{C}([0,T];\mathbb{L}^2(\Omega))\cap\mathrm{L}^2(0,T;\mathbb{H}_0^1(\Omega)))\times\mathrm{C}([0,T];\mathrm{L}^2(\Omega)),$$ with $$(\partial_t\p,\partial_t \varphi)\in\mathrm{L}^2(0,T;\mathbb{H}^{-1}(\Omega))\times\mathrm{L}^2(0,T;\mathrm{L}^2(\Omega)).$$
\end{theorem}

\subsection{First order necessary conditions of optimality}
In this subsection, we prove the first order necessary condition or optimality for the optimal control problem \eqref{control problem} and discuss about Pontryagin's maximum principle. We characterize the optimal control obtained in Theorem \ref{optimal} in terms of the adjoint variable. Remember that our optimal control problem is a minimization of the cost functional given in (\ref{cost}) and hence we obtain  a  minimum principle. We mainly follow the  techniques used in the works, \cite{FART,sritharan} (incompressible Navier-Stokes equations), \cite{BDM} (Cahn-Hilliard-Navier-Stokes equations), etc to obtain the first order necessary conditions. The main result of our paper is:
\begin{theorem}\label{main}
	Let $(\u^*,\xi^*,\U^*)\in\mathscr{A}_{\mathrm{ad}}$ be the optimal solution of the problem \eqref{control problem} obtained in Theorem \ref{optimal}. Then, there exists \emph{a unique weak solution} $(\p,\varphi)\in(\mathrm{C}([0,T];\mathbb{L}^2(\Omega))\cap\mathrm{L}^2(0,T;\mathbb{H}_0^1(\Omega)))\times\mathrm{C}([0,T];\mathrm{L}^2(\Omega))$ of the adjoint system \eqref{adj} such that
	\begin{eqnarray}\label{3.41}
	\U^*(t)=\Delta\p(t)\in \H^{-1}(\Omega),  \text{ a.e. }  t \in [0,T].
	\end{eqnarray}
\end{theorem}
From \eqref{40b}, we know that 
\begin{align}\label{40a}
&\sup_{t\in[0,T]}\left(\|\p(t)\|_{\mathbb{L}^2}^2+\|\varphi(t)\|_{\mathrm{L}^2}^2\right)+\alpha\int_0^T\|\nabla\p(t)\|_{\mathbb{L}^2}^2\d t\nonumber\\&\leq \bigg( \|\u^*(T)-\u^f\|_{\mathbb{L}^2}^2+\|\xi^*(T)-\xi^f\|_{\mathrm{L}^2}^2+\int_0^T\|\u^*(t)-\u_d(t)\|_{\mathbb{L}^2}^2\d t\nonumber\\&\quad+\int_0^T\|\xi^*(t)-\xi_d(t)\|_{\mathrm{L}^2}^2\d t+\frac{4}{\alpha}\int_0^T\|\nabla\u^*(t)\|_{\L^2}^2\d t\bigg)e^{2\left[M+\frac{8\mu^2}{\alpha}+4\left(\frac{1}{\alpha}+1\right)\right]T}=:K_T. 
\end{align}
Before proving Theorem \ref{main}, we provide two important Lemmas, which is used to establish our main Theorem. 
\begin{lemma}\label{lem3.7}
	Let $(\u^*,\xi^*,\U^*)\in\mathscr{A}_{\mathrm{ad}}$ be an optimal triplet for the control problem \eqref{control problem}. Let $(\u_{\U^*+\tau\U},\xi_{\U^*+\tau\U})$ be the unique weak solution of the system \eqref{a1} with the control $\U^*+\tau\U$, for $\tau\in\R$ and $\U\in\H^{-1}(\Omega)$. Then, we have 
	\begin{align}\label{325a}
	&\sup_{t\in[0,T]}\left[\|\u_{\U^*+\tau\U}(t)-\u_{\U^*}(t)\|_{\mathbb{L}^2}^2+\|\xi_{\U^*+\tau\U}(t)-\xi^*(t)\|_{\mathrm{L}^2}^2\right]\nonumber\\&\quad+\alpha\int_0^T\|\nabla(\u_{\U^*+\tau\U}(t)-\u_{\U^*}(t))\|_{\mathbb{L}^2}^2\d t\nonumber\\&\quad\leq \left\{\frac{8\tau^2}{\alpha}\int_0^T\|\U(t)\|_{\H^{-1}}^2\d t\right\}e^{ \left[\frac{4}{\alpha}\left(2+\mu^2\right)+M\right]T}.
	\end{align}
\end{lemma}
\begin{proof}
	Since $(\u_{\U^*+\tau \U},\xi_{\U^*+\tau \U})$ and $(\u_{\U^*},\xi_{\U^*})$ are the unique weak solutions of the system \eqref{a1} corresponding to the  controls $\U^*+\tau \U$ and $\U^*$, respectively, we know that $(\widetilde{\u},\widetilde{\xi})=(\u_{\U^*+\tau \U}-\u_{\U^*},\xi_{\U^*+\tau \U}-\xi_{\U^*})$  satisfy the following system:
	\begin{equation}\label{11}
	\left\{
	\begin{aligned}
	\frac{\partial \widetilde{\mathbf{u}}(t)}{\partial t}+\A \widetilde{\mathbf{u}}(t)+\B(\u_{\U^*+\tau \U}(t))-\B(\u_{\U^*}(t))+\nabla  \widetilde{\xi} (t)&=\tau \U(t),\ \text{ in }\ \H^{-1}(\Omega),\\
	\frac{\partial \widetilde{\xi}(t)}{\partial t}+\mathrm{div}(h \widetilde{\u}(t))&=0, \ \text{ in }\ \mathrm{L}^2(\Omega),\\
	\widetilde{\u}(0)&=\mathbf{0},\ \widetilde{\xi}(0)=0, 
	\end{aligned}
	\right.
	\end{equation}
	for a.e. $t\in[0,T]$. 
	Let us take inner product with $\widetilde{\u}(\cdot)$ to the first equation in \eqref{11} to obtain 
	\begin{align}\label{322}
	\frac{1}{2}&\frac{\d }{\d t}\|\widetilde{\u}(t)\|_{\mathbb{L}^2}^2+\alpha\|\nabla\widetilde{\u}(t)\|_{\mathbb{L}^2}^2\nonumber\\&=-(\B(\u_{\U^*+\tau \U}(t))-\B(\u_{\U^*}(t)),\widetilde{\u}(t))_{\mathbb{L}^2}-\langle\nabla  \widetilde{\xi} (t),\widetilde{\u}(t)\rangle+\tau\langle\U(t),\widetilde{\u}(t)\rangle.
	\end{align}
	Using an integration by parts, H\"older's and Young's inequalities, we estimate the final two terms from the right hand side of the equality \eqref{322} as 
	\begin{align}
	-\langle\nabla  \widetilde{\xi} ,\widetilde{\u}\rangle+\tau\langle\U,\widetilde{\u}\rangle	&= (\widetilde{\xi},\mathrm{div\ }\widetilde{\u})_{\mathrm{L}^2}+\tau\langle\U,\widetilde{\u}\rangle\nonumber\\&\leq \|\widetilde{\xi}\|_{\mathrm{L}^2}\|\mathrm{div\ }\widetilde{\u}\|_{\mathrm{L}^2}+\tau\|\U\|_{\H^{-1}}\|\nabla\widetilde{\u}\|_{\mathbb{L}^2}\nonumber\\&\leq \frac{\alpha}{4}\|\nabla\widetilde{\u}\|_{\mathbb{L}^2}^2+\frac{4}{\alpha}\|\widetilde{\xi}\|_{\mathrm{L}^2}^2+\frac{2\tau^2}{\alpha}\|\U\|_{\H^{-1}}^2. 
	\end{align}
	We now take inner product with $\widetilde{\xi}(\cdot)$ to the second equation in \eqref{11} to find 
	\begin{align}\label{323}
	\frac{1}{2}\frac{\d}{\d t}\|\widetilde{\xi}(t)\|_{\mathrm{L}^2}^2&=-(h\mathrm{div\ }\widetilde{\u}(t),\widetilde{\xi}(t))_{\mathrm{L}^2}-(\nabla h\cdot\widetilde{\u}(t),\widetilde{\xi}(t))_{\mathrm{L}^2}\nonumber\\&\leq \|h\mathrm{div\ }\widetilde{\u}(t)\|_{\mathrm{L}^2}\|\widetilde{\xi}(t)\|_{\mathrm{L}^2}+\|\nabla h\cdot\widetilde{\u}(t)\|_{\mathrm{L}^2}\|\widetilde{\xi}(t)\|_{\mathrm{L}^2}\nonumber\\&\leq \|h\|_{\mathrm{L}^{\infty}}\|\mathrm{div\ }\widetilde{\u}(t)\|_{\mathrm{L}^2}\|\widetilde{\xi}(t)\|_{\mathrm{L}^2}+\|\nabla h\|_{\mathbb{L}^{\infty}}\|\widetilde{\u}(t)\|_{\mathbb{L}^2}\|\widetilde{\xi}(t)\|_{\mathrm{L}^2}\nonumber\\&\leq \frac{\alpha}{4}\|\nabla\widetilde{\u}(t)\|_{\mathbb{L}^2}^2+\left(\frac{2\mu^2}{\alpha}+\frac{M}{2}\right)\|\widetilde{\xi}(t)\|_{\mathrm{L}^2}^2+\frac{M}{2}\|\widetilde{\u}(t)\|_{\mathbb{L}^2}^2,
	\end{align}
	where we used the Cauchy-Schwarz inequality, Holder's and Young's inequalities. Adding \eqref{322} and \eqref{323} together, we get 
	\begin{align*}
	&\frac{\d}{\d t}\left(\|\widetilde{\u}(t)\|_{\mathbb{L}^2}^2+\|\widetilde{\xi}(t)\|_{\mathrm{L}^2}^2\right)+\alpha\|\nabla\widetilde{\u}(t)\|_{\mathbb{L}^2}^2+(\B(\u_{\U^*+\tau \U}(t))-\B(\u_{\U^*}(t)),\widetilde{\u}(t))_{\mathbb{L}^2}\nonumber\\&\leq M\|\widetilde{\u}(t)\|_{\mathbb{L}^2}^2+ \left[\frac{4}{\alpha}\left(2+\mu^2\right)+M\right]\|\widetilde{\xi}(t)\|_{\mathrm{L}^2}^2+\frac{4\tau^2}{\alpha}\|\U(t)\|_{\H^{-1}}^2.
	\end{align*}
	Integrating the above inequality from $0$ to $t$, we obtain 
	\begin{align}\label{324}
	&\|\widetilde{\u}(t)\|_{\mathbb{L}^2}^2+\|\widetilde{\xi}(t)\|_{\mathrm{L}^2}^2+\alpha\int_0^t\|\nabla\widetilde{\u}(s)\|_{\mathbb{L}^2}^2\d s\nonumber\\&\quad+2\int_0^t(\B(\u_{\U^*+\tau \U}(s))-\B(\u_{\U^*}(s)),\widetilde{\u}(s))_{\mathbb{L}^2}\d s\nonumber\\&\leq \left[\frac{4}{\alpha}\left(2+\mu^2\right)+M\right]\int_0^t\left(\|\widetilde{\u}(s)\|_{\mathbb{L}^2}^2+\|\widetilde{\xi}(s)\|_{\mathrm{L}^2}^2\right)\d s+\frac{4\tau^2}{\alpha}\int_0^t\|\U(s)\|_{\H^{-1}}^2\d s.
	\end{align}
	Using the Lemma \ref{lem2.3} (iii), we know that 
	\begin{align}
	\int_0^t(\B(\u_{\U^*+\tau \U}(s))-\B(\u_{\U^*}(s)),\widetilde{\u}(s))_{\mathbb{L}^2}\d s\geq 0,\ \text{ for all }\ t\in[0,T]. 
	\end{align}
	An application of Gronwall's inequality in \eqref{324} yields 
	\begin{align}\label{325}
	&\|\widetilde{\u}(t)\|_{\mathbb{L}^2}^2+\|\widetilde{\xi}(t)\|_{\mathrm{L}^2}^2\leq\left\{\frac{4\tau^2}{\alpha}\int_0^T\|\U(t)\|_{\H^{-1}}^2\d t\right\}e^{ \left[\frac{4}{\alpha}\left(2+\mu^2\right)+M\right]T},
	\end{align}
	for all $t\in[0,T]$, which completes the proof. 
\end{proof}
The following lemma gives the differentiability of the mapping $\U \mapsto (\u_{\U},\xi_{\U})$  from $\mathscr{U}_{\mathrm{ad}}$ into   $\left(\mathrm{C}([0,T];\L^2(\Omega))\cap\mathrm{L}^{2}(0,T;\H_0^1(\Omega))\right) \times \mathrm{C}([0,T];\mathrm{L}^2(\Omega))$.
\begin{lemma}\label{lem3.8}
	Let $(\u_0, \xi_0)\in(\mathbb{L}^2(\Omega)\times\mathrm{L}^2(\Omega))$ and $\f\in\mathrm{L}^2(0,T;\mathbb{H}^{-1}(\Omega))$ be given. Then the mapping $\U \mapsto (\u_{\U},\xi_{\U})$  from $\mathscr{U}_{\mathrm{ad}}=\mathrm{L}^{2}(0,T;\mathbb{H}^{-1}(\Omega))$ into the function space $\left(\mathrm{C}([0,T];\L^2(\Omega))\cap\mathrm{L}^{2}(0,T;\H_0^1(\Omega))\right) \times \mathrm{C}([0,T];\mathrm{L}^2(\Omega))$ is Gateaux differentiable. Furthermore, we have 
	\begin{align}
	\label{limit}
	\left\{\lim_{\tau\downarrow 0}\frac{\u_{\U^*+\tau \U} - \u_{\U^*}}{\tau}, \lim_{\tau\downarrow 0}\frac{\xi_{\U^*+\tau \U} - \xi_{\U^*}}{\tau} \right\}= \{\mathfrak{w} , \eta\}, \end{align}
	where $(\mathfrak{w}, \eta)$ is the \emph{unique weak solution} of the linearized system:
	\begin{equation}\label{4a}
	\left\{
	\begin{aligned}
	\frac{\partial \mathfrak{w}(t)}{\partial t}+\A\mathfrak{w}(t)+\B'(\mathbf{u}_{\U^*}(t))\mathfrak{w}(t)+\nabla \eta (t)&=\U(t),\ \text{ in }\ \H^{-1}(\Omega),\\
	\frac{\partial \eta(t)}{\partial t}+\mathrm{div}(h \mathfrak{w}(t))&=0, \ \text{ in }\ \mathrm{L}^2(\Omega),\\
	\mathfrak{w}(0)&=\mathbf{0},\ \eta(0)=0, 
	\end{aligned}
	\right.
	\end{equation}
	for a.e. $t\in[0,T]$ and the pairs $(\u_{\U^*},\xi_{\U^*})$  and $(\u_{\U^*+\tau\U},\xi_{\U^*+\tau\U})$ are the unique weak solutions of the controlled system (\ref{1}) with the controls $\U^*$ and $\U^*+\tau\U$, respectively.
	That is, we have 
	\begin{equation}\label{lim}
	\left\{
	\begin{aligned}
	\lim_{\tau\downarrow 0}\frac{\|\u_{\U^*+\tau \U} (t)- \u_{\U^*}(t)-\tau\mathfrak{w}(t)\|_{\L^2}}{\tau}&=0, \ \text{ for all }\ t\in[0,T], \\
	\lim_{\tau\downarrow 0}\frac{\|\u_{\U^*+\tau \U} - \u_{\U^*}-\tau\mathfrak{w}\|_{\mathrm{L}^{2}(0,T;\H_0^1(\Omega))}}{\tau}&=0,\\
	\lim_{\tau\downarrow 0}\frac{\|\xi_{\U^*+\tau \U} (t)- \xi_{\U^*}(t)-\tau\eta(t)\|_{\mathrm{L}^{2}}}{\tau}&=0, \ \text{ for all }\ t\in[0,T].
	\end{aligned}
	\right.
	\end{equation}
\end{lemma}

\begin{proof}
	Let us define 
	\begin{align*}
	(\y,\varrho):=(\u_{\U^*+\tau\U}-\u_{\U^*}-\tau\mathfrak{w},\xi_{\U^*+\tau\U}-\xi_{\U^*}-\tau\eta).
	\end{align*}
	Then $	(\y,\varrho)$ satisfies the following system:
	\begin{equation}\label{4b}
	\left\{
	\begin{aligned}
	&	\frac{\partial \y(t)}{\partial t}+\A\y(t)+\B'(\mathbf{u}_{\U^*}(t))\y(t)+\nabla \varrho (t)\\&=\B'(\u_{\U^*}(t))(\u_{\U^*+\tau\U}(t)-\u_{\U^*}(t)) -\B(\u_{\U^*+\tau\U}(t))+\B(\u_{\U^*}(t)),\ \text{ in }\ \H^{-1}(\Omega),\\
	&	\frac{\partial \varrho(t)}{\partial t}+\mathrm{div}(h \y(t))=0, \ \text{ in }\ \mathrm{L}^2(\Omega),\\
	&	\y(0)=\mathbf{0},\ \varrho(0)=0, 
	\end{aligned}
	\right.
	\end{equation}
	for a.e. $t\in[0,T]$. Remember that the the term $\B'(\u_{\U^*})(\u_{\U^*+\tau\U}-\u_{\U^*}) -\B(\u_{\U^*+\tau\U})+\B(\u_{\U^*})\in\mathrm{L}^2(0,T;\L^2(\Omega))$ and hence the system \eqref{4b} has a unique weak solution with $$	(\y,\varrho)\in(\mathrm{C}([0,T];\mathbb{L}^2(\Omega))\cap\mathrm{L}^2(0,T;\mathbb{H}_0^1(\Omega)))\times\mathrm{C}([0,T];\mathrm{L}^2(\Omega)).$$	Using Taylor's series (see Theorem 7.9.1, \cite{PGC}), the right hand side of the first equation in \eqref{4b} can be written as 
	\begin{align*}
	&\|\mathbf{N}(\u_{\U^*},\u_{\U^*+\tau\U})\|_{\L^2}\nonumber\\&:=\|	\B'(\u_{\U^*})(\u_{\U^*+\tau\U}-\u_{\U^*}) -\B(\u_{\U^*+\tau\U})+\B(\u_{\U^*})\|_{\L^2}\nonumber\\&=	\left\|\left[\B'(\u_{\U^*})-\int_0^1\B'(\theta\u_{\U^*+\tau\U}+(1-\theta)\u_{\U^*})\d\theta\right](\u_{\U^*+\tau\U}-\u_{\U^*})\right\|_{\L^2}\nonumber\\&=\left\| 2\gamma\int_0^1\left[|\u_{\U^*}+\w^0|-|\theta\u_{\U^*+\tau\U}+(1-\theta)\u_{\U^*}+\w^0|\right]\d\theta(\u_{\U^*+\tau\U}-\u_{\U^*})]\right\|_{\L^2}\nonumber\\&\leq 2\left\|\gamma\int_0^1|\u_{\U^*}+\w^0-(\theta\u_{\U^*+\tau\U}+(1-\theta)\u_{\U^*}+\w^0)|\d\theta|\u_{\U^*+\tau\U}-\u_{\U^*}|\right\|_{\L^2}\nonumber\\&\leq 2\|\gamma\|_{\mathrm{L}^{\infty}}\sup_{0<\theta<1}|\theta|\||\u_{\U^*+\tau\U}-\u_{\U^*}|^2\|_{\L^2}\nonumber\\&\leq \frac{2r}{\lambda}\|\u_{\U^*+\tau\U}-\u_{\U^*}\|_{\mathbb{L}^4}^2\nonumber\\&\leq \frac{2\sqrt{2}r}{\lambda}\|\u_{\U^*+\tau\U}-\u_{\U^*}\|_{\mathbb{L}^2}\|\nabla(\u_{\U^*+\tau\U}-\u_{\U^*})\|_{\mathbb{L}^2},
	\end{align*}
	and 
	\begin{align}\label{3.30}
	&\|\mathbf{N}(\u_{\U^*},\u_{\U^*+\tau\U})\|_{\mathrm{L}^2(0,T;\L^2(\Omega))}\nonumber\\&\leq \frac{2\sqrt{2}r}{\lambda}\|\u_{\U^*+\tau\U}-\u_{\U^*}\|_{\mathrm{L}^{\infty}(0,T;\mathbb{L}^2(\Omega))}\|\nabla(\u_{\U^*+\tau\U}-\u_{\U^*})\|_{\mathrm{L}^{2}(0,T;\mathbb{L}^2(\Omega))}\nonumber\\&\leq \left\{\frac{8\sqrt{2}\tau^2r}{\alpha\lambda}\int_0^T\|\U^*(t)\|_{\H^{-1}}^2\d t\right\}e^{ \left[\frac{8}{\alpha}\left(2+\mu^2\right)+M\right]T},
	\end{align}
	using \eqref{325a}. A calculation similar to \eqref{218} gives 
	\begin{align}\label{330}
	&\|\y(t)\|_{\mathbb{L}^2}^2+\|\varrho(t)\|_{\mathrm{L}^2}^2+\alpha\int_0^t\|\nabla\y(s)\|_{\mathbb{L}^2}^2\d s\nonumber\\&\leq \int_0^t	\|\mathbf{N}(\u_{\U^*}(s),\u_{\U^*+\tau\U}(s))\|_{\mathbb{L}^2}^2\d s+K_1\int_0^t\left(\|\y(s)\|_{\mathbb{L}^2}^2+\|\varrho(s)\|_{\mathrm{L}^2}^2\right)\d s,
	\end{align}
	where $K_1=\max\left\{M+1,M+\frac{4\mu^2}{\alpha}\right\}$. An application of Gronwall's inequality in \eqref{330}  gives
	\begin{align}\label{332}
	&\|\y(t)\|_{\mathbb{L}^2}^2+\|\varrho(t)\|_{\mathrm{L}^2}^2+\alpha\int_0^t\|\nabla\y(s)\|_{\mathbb{L}^2}^2\d s\nonumber\\& \leq \left\{\int_0^T	\|\mathbf{N}(\u_{\U^*}(t),\u_{\U^*+\tau\U}(t))\|_{\mathbb{L}^2}^2\d t\right\}e^{K_1T},
	\end{align}
	for all  $t\in[0,T]$. Using \eqref{3.30} in \eqref{330}, it can be easily seen that 
	\begin{align}\label{333}
	\|\u_{\U^*+\tau \U}(t) - \u_{\U^*}(t)-\tau\mathfrak{w}(t)\|_{\L^2}&\leq \|\mathbf{N}(\u_{\U^*},\u_{\U^*+\tau\U})\|_{\mathrm{L}^2(0,T;\L^2(\Omega))} e^{\frac{K_1T}{2} }\nonumber\\&\leq \left\{\frac{8\sqrt{2}\tau^2r}{\alpha\lambda}\int_0^T\|\U^*(t)\|_{\H^{-1}}^2\d t\right\}e^{ \frac{(K_1+K_2)T}{2}},
	\end{align}
	for all $t\in[0,T]$,	where $K_2=\left[\frac{8}{\alpha}\left(2+\mu^2\right)+M\right]$. Then, we have 
	\begin{align}
	&	\lim_{\tau\downarrow 0}\frac{\|\u_{\U^*+\tau \U} (t)- \u_{\U^*}(t)-\tau\mathfrak{w}(t)\|_{\L^2}}{\tau}\nonumber\\&\leq \lim_{\tau\downarrow 0}\left\{\frac{8\sqrt{2}\tau r}{\alpha\lambda}\int_0^T\|\U^*(t)\|_{\H^{-1}}^2\d t\right\}e^{ \frac{(K_1+K_2)T}{2}}=0, 
	\end{align}
	for all $t\in[0,T]$	and 
	\begin{align}
	\lim_{\tau\downarrow 0}\frac{\|\u_{\U^*+\tau \U} - \u_{\U^*}-\tau\mathfrak{w}\|_{\mathrm{L}^2(0,T;\H_0^1(\Omega))}}{\tau}=0.
	\end{align}
	Using \eqref{332}, a calculation similar to \eqref{333} also shows that $$\lim_{\tau\downarrow 0}\frac{\|\xi_{\U^*+\tau \U} (t)- \xi_{\U^*}(t)-\tau\eta(t)\|_{\mathrm{L}^2}}{\tau}=0,$$ for all $t\in[0,T]$, which completes the proof. 
\end{proof}
Next, we prove our main result (Theorem \ref{main}) using Lemmas \ref{lem3.7} and \ref{lem3.8}. Note that the norms appearing in the cost functional \eqref{cost} are Gateaux differentiable and hence we don't need to consider any approximate cost functional or methods which use needle or spike perturbations in this work.

\begin{proof}[Proof of Theorem \ref{main}]
	Let $(\u^*,\xi^*,\U^*)\in\mathscr{A}_{\mathrm{ad}}$ be the optimal triplet of the control problem \eqref{control problem} obtained in Theorem \ref{optimal}.  Let $\mathcal{G}(\U)=\mathscr{J}(\u_{\U},\xi_{\U},{\U})$, where $({\u}_{\U},\xi_{\U},\U)$ is the unique weak solution of the controlled system \eqref{a1} with the control $\U\in\H^{-1}(\Omega)$, for a.e. $t\in[0,T]$.  Note that $\U^*+\tau\U\in\mathscr{U}_{\mathrm{ad}}$, for all $\tau$ and $(\u_{\U^*+\tau\U},\xi_{\U^*+\tau\U},{\U^*+\tau\U})\in\mathscr{A}_{\mathrm{ad}}$. Then, we have
	\begin{align*}
	&\mathcal{G}(\U^*+ \tau \U) - \mathcal{G}(\U^*) \\&= \mathscr{J}(\u_{\U^* + \tau \U},\xi_{\U^* + \tau \U},{\U^* + \tau \U})- \mathscr{J}(\u_{\U^*},\xi_{\U^* },{\U^*})  \\       
	&= \frac{1}{2} \int_0^T  \|\u_{\U^* + \tau \U}(t)-\u_d(t)\|^2_{\L^2} \d t+\int_0^T\|\nabla\u_{\U^* + \tau \U}(t)\|_{\L^2}^2\d t\nonumber\\&\quad+ \frac{1}{2} \int_0^T  \|\xi_{\U^* + \tau \U}(t)-\xi_d(t)\|^2_{\mathrm{L}^2} \d t+\frac{1}{2} \int_0^T  \|{\U^*(t) + \tau \U(t)}\|^2_{\H^{-1}} \d t\nonumber\\&\quad+\frac{1}{2}\|\u_{\U^* + \tau \U}(T)-\u^f\|^2_{\L^2}+\frac{1}{2}\|\xi_{\U^* + \tau \U}(T)-\xi^f\|^2_{\mathrm{L}^2},\nonumber  \\
	& \quad -\frac{1}{2} \int_0^T  \|\u_{\U^*}(t) -\u_d(t)\|^2_{\L^2}\d t-\frac{1}{2}\int_0^T\|\nabla\u_{\U^*}(t)\|_{\L^2}^2\d t \nonumber\\&\quad- \frac{1}{2} \int_0^T  \|\xi_{\U^*}(t) -\xi_d(t)\|^2_{\L^2} \d t - \frac{1}{2} \int_0^T  \| \U^*(t)\|^2_{\H^{-1}} \d t-\frac{1}{2}\|\u_{\U^*}(T)-\u^f\|^2_{\L^2}\nonumber\\&\quad-\frac{1}{2}\|\xi_{\U^* }(T)-\xi^f\|^2_{\mathrm{L}^2}
	\\ &= \frac{1}{2}\int_0^T (\u_{\U^* + \tau \U}(t) - \u_{\U^*}(t),\u_{\U^*+\tau \U}(t) + \u_{\U^*}(t) -2\u_d(t))_{\mathbb{L}^2}\d t \nonumber\\&\quad+\int_0^T(\nabla(\u_{\U^*+\tau\U}(t)-\u_{\U^*}(t)),\nabla(\u_{\U^*+\tau\U}(t)+\u_{\U^*}(t)))_{\L^2}\d t  \\&
	\quad + \frac{1}{2}\int_0^T (\xi_{\U^* + \tau \U} (t)- \xi_{\U^*}(t),\xi_{\U^*+\tau \U}(t) + \xi_{\U^*}(t)-2\xi_d(t))_{\mathrm{L}^2}\d t \\
	&\quad  + \frac{1}{2} \int_0^T  \|(-\Delta)^{-1/2}(\U^*(t) + \tau \U(t))\|^2_{\L^{2}} \d t - \frac{1}{2} \int_0^T  \| (-\Delta)^{-1/2}\U^*(t)\|^2_{\L^{2}}\d t \nonumber\\&\quad+\frac{1}{2}(\u_{\U^* + \tau \U}(T)-\u_{\U^* }(T),\u_{\U^* + \tau \U}(T)+\u_{\U^*}(T)-2\u^f)_{\L^2}\nonumber\\&\quad+\frac{1}{2}(\xi_{\U^* + \tau \U}(T)-\xi_{\U^* }(T),\xi_{\U^* + \tau \U}(T)+\xi_{\U^*}(T)-2\xi^f)_{\mathrm{L}^2} .
	\end{align*}
	From the above calculations, it is immediate that 
	\begin{align}\label{lam}
	&	\mathcal{G}(\U^*+ \tau \U) - \mathcal{G}(\U^*)\nonumber \\&= \frac{1}{2}\int_0^T  \|\u_{\U^*+\tau \U} (t)- \u_{\U^*}(t)\|^2_{\mathbb{L}^2} \d t+ \int_0^T (\u_{\U^*+\tau \U} (t)- \u_{\U^*}(t),\u_{\U^*}(t) -\u_d(t))_{\mathbb{L}^2}\d t \nonumber\\&\quad+\frac{1}{2}\int_0^T\|\nabla(\u_{\U^*+\tau \U} (t)- \u_{\U^*}(t))\|^2_{\L^2} \d t-\int_0^T\langle\u_{\U^*+\tau \U} (t)- \u_{\U^*}(t),\Delta\u_{\U^*}(t)\rangle\d t \nonumber \\
	&\quad+\frac{1}{2}\int_0^T  \|\xi_{\U^*+\tau \U}(t) - \xi_{\U^*}(t)\|^2_{\mathrm{L}^2} \d t+ \int_0^T (\xi_{\U^*+\tau \U}(t) - \xi_{\U^*}(t),\xi_{\U^*} (t)-\xi_d(t))_{\mathrm{L}^2}\d t\nonumber\\
	& \quad+  \int_0^T \tau ((-\Delta)^{-1/2}\U(t),(-\Delta)^{-1/2}\U^*(t))_{\mathbb{L}^2} \d t+ \frac{\tau^2}{2} \int_0^T  \|(-\Delta)^{-1/2}\U(t)\|_{\mathbb{L}^2}^2   \d t \nonumber\\&\quad+\frac{1}{2}\|\u_{\U^*+\tau \U} (T)- \u_{\U^*}(T)\|^2_{\L^2}+(\u_{\U^* + \tau \U}(T)-\u_{\U^* }(T),\u_{\U^*}(T)-\u^f)_{\L^2} \nonumber\\&\quad+\frac{1}{2}\|\xi_{\U^*+\tau \U} (T)- \xi_{\U^*}(T)\|^2_{\mathrm{L}^2}+(\xi_{\U^* + \tau \U}(T)-\xi_{\U^* }(T),\xi_{\U^*}(T)-\u^f)_{\mathrm{L}^2} .
	\end{align}
	Using the estimate  \eqref{325a} (see Lemma \ref{lem3.7}), we know that $\|\u_{\U^*+\tau \U} - \u_{\U^*}\|^2_{\mathrm{L}^2(0,T;\L^2(\Omega))}$,  $\|\xi_{\U^*+\tau \U} - \xi_{\U^*}\|_{\mathrm{L}^2(0,T;\mathrm{L}^2(\Omega))}^2$, $\|\u_{\U^*+\tau \U} (T)- \u_{\U^*}(T)\|^2_{\L^2}$ and  $\|\xi_{\U^*+\tau \U} (T)- \xi_{\U^*}(T)\|^2_{\mathrm{L}^2}$  can be estimated by $C\tau^2 \|\U\|^2_{\mathrm{L}^{2}(0,T;\H^{-1}(\Omega))}$. Thus dividing by $\tau$, and then sending $\tau \downarrow 0$, we easily have  $\|\u_{\U^*+\tau \U} - \u_{\U^*}\|_{\mathrm{L}^2(0,T;\L^2(\Omega))} $,   $\|\xi_{\U^*+\tau \U} - \xi_{\U^*}\|_{\mathrm{L}^2(0,T;\mathrm{L}^2(\Omega))} $, $\|\u_{\U^*+\tau \U} (T)- \u_{\U^*}(T)\|_{\L^2}$ and  $\|\xi_{\U^*+\tau \U} (T)- \xi_{\U^*}(T)\|_{\mathrm{L}^2}$  $ \rightarrow 0$ as $\tau \downarrow 0$. 
	
	Let us denote the Gateaux derivative of $\mathcal{G}$ at $\U^*$ in the direction of $\U\in\H^{-1}(\Omega)$ by $\langle\mathcal{G}'(\U^*),\U\rangle$. Let $(\mathfrak{w},\eta)$ satisfy the linearized system \eqref{4a} with the control $\U$.  From Lemma \ref{lem3.8}, we also have the convergences (see \eqref{lim}): 
	\begin{align*}
	&	\lim_{\tau\downarrow 0}\frac{\|\u_{\U^*+\tau \U} (t)- \u_{\U^*}(t)-\tau\mathfrak{w}(t)\|_{\L^2(\Omega)}}{\tau}=0,\\ & \text{ and }\
	\lim_{\tau\downarrow 0}\frac{\|\xi_{\U^*+\tau \U} (t)- \xi_{\U^*}(t)-\tau\eta(t)\|_{\mathrm{L}^{2}}}{\tau}=0, 
	\end{align*}
	for all $t\in[0,T],$ and 
	\begin{align*}
	\lim_{\tau\downarrow 0}\frac{\|\u_{\U^*+\tau \U} - \u_{\U^*}-\tau\mathfrak{w}\|_{\mathrm{L}^2(0,T;\H_0^1(\Omega))}}{\tau}=0.
	\end{align*}
	Dividing by $\tau$ and then taking $\tau \downarrow 0$ in (\ref{lam}), we obtain
	\begin{align}\label{329}
	0 &\leq \langle\mathcal{G}'(\U^*), \U\rangle = \underset{\tau \downarrow 0}{\lim} \frac{\mathcal{G}(\U^*+ \tau \U)- \mathcal{G}(\U^*)}{\tau} \nonumber\\
	&= \int_0^T  (\mathfrak{w}(t),\u_{\U^*}(t) -\u_d(t))_{\mathbb{L}^2}\d t-\int_0^T\langle\mathfrak{w}(t),\Delta\u_{\U^*}(t)\rangle\d t \nonumber\\&\quad+ \int_0^T   (\eta(t),\xi_{\U^*}(t) -\xi_d(t))_{\mathrm{L}^2} \d t +  \int_0^T  ((-\Delta)^{-1/2}\U(t),(-\Delta)^{-1/2}\U^*(t))_{\mathbb{L}^2} \d t\nonumber\\&\quad+(\mathfrak{w}(T),\u_{\U^*}(T)-\u^f)_{\L^2}+(\eta(T),\xi_{\U^*}(T)-\xi^f)_{\mathrm{L}^2}
	\nonumber\\&= \int_0^T  \langle \mathfrak{w}(t),-\p_t(t)+\widetilde{\A}\p(t)+\B'(\u_{\U^*}(t))\p(t) -h\nabla\varphi(t)\rangle\d t \nonumber\\&\quad+\int_0^T  (\eta(t), -\varphi_t(t)  -\mathrm{div\ }\p(t))_{\mathrm{L}^2}\d t  +  \int_0^T  ((-\Delta)^{-1/2}\U(t),(-\Delta)^{-1/2}\U^*(t))_{\mathbb{L}^2} \d t\nonumber\\&= \int_0^T\langle\mathfrak{w}_t(t)+\A\mathfrak{w}(t)+\B'(\u_{\U^*}(t))\mathfrak{w}(t)+\nabla \eta (t),\p(t)\rangle\d t\nonumber\\&\quad+\int_0^T(\eta_t(t)+\mathrm{div}(h \mathfrak{w}(t)),\varphi(t))_{\mathrm{L}^2}+  \int_0^T  ((-\Delta)^{-1/2}\U(t),(-\Delta)^{-1/2}\U^*(t))_{\mathbb{L}^2} \d t\nonumber\\&= \int_0^T\langle \U (t),\p(t)\rangle\d t  +  \int_0^T  \langle\U(t),(-\Delta)\U^*(t)\rangle \d t,
	\end{align}  
	where we used the adjoint system \eqref{adj} and performed an integration by parts.  We also
	used the equation satisfied by $(\mathfrak{w},\eta)$ with the control $\U$ (see \eqref{4a}). Thus, from \eqref{329}, we infer that  
	$$0 \leq \langle\mathcal{G}'(\U^*(t)), \U(t)\rangle= \int_0^T\langle\U (t),\p(t)\rangle\d t  +  \int_0^T  \langle \U(t),(-\Delta)^{-1}\U^*(t)\rangle \d t.$$ 
	Similarly, if we take the directional derivative of $\mathcal{G}$ in the direction of $-\U\in\H^{-1}(\Omega)$, we  obtain  $\langle \mathcal{G}'(\U^*(t)), \U(t)\rangle \leq 0.$
	Hence, we obtain 
	$\langle \mathcal{G}'(\U^*(t)) , \U(t) \rangle= 0,$ and we infer that 
	\begin{align}\label{4.51}
	\int_0^T\langle\U (t),\p(t)+(-\Delta)^{-1}\U^*(t)\rangle \d t=0, \end{align} for all $\U\in\H^{-1}(\Omega)$. Thus, it is immediate that 
	\begin{align*}
	\langle\U(t),\p(t) + (-\Delta)^{-1}\U^*(t)\rangle &= 0, \ \text{ for all } \ \U \in \H^{-1}(\Omega),\  \text{  and  a.e. }\  t \in [0,T].
	\end{align*}
	Since the above equality is true for all $\U \in \H^{-1}(\Omega)$, we get $\p(t) +(-\Delta)^{-1} \U^*(t)=0$, a.e. $t\in[0,T]$ and hence we obtain   
	$
	\U^*(t)=\Delta\p(t)\in \H^{-1}(\Omega),  \text{ a.e. }  t \in [0,T].
	$ Thus,  there exists \emph{a unique weak solution} $(\p,\varphi)$ of the adjoint system \eqref{adj} such that for a.e. $t \in [0,T]$,  \eqref{3.41} is satisfied.
\end{proof}	

\begin{remark}
	(1). Let us formally discuss about the Pontryagin minimum principle in terms of the \emph{Hamiltonian formulation} and obtain the result \eqref{3.41}. 	We define the \emph{Lagrangian} as
	$$\mathscr{L}(\u,\xi,\U) := \frac{1}{2} (\|\u-\u_d\|^2_{\mathbb{L}^2} + \|\xi - \xi_d\|^2_{\mathrm{L}^2}+\|\nabla\u\|_{\L^2}^2+ \|\U\|^2_{\mathbb{H}^{-1}}), $$ so that $\mathscr{J}(\u,\xi,\U)=\int_0^T\mathscr{L}(\u(t),\xi(t),\U(t))\d t+\frac{1}{2}\|\u(T)-\u^f\|^2_{\L^2}+\frac{1}{2}\|\xi(T)-\xi^f\|^2_{\mathrm{L}^2}$. 
	We define the corresponding \emph{Hamiltonian} by
	$$\mathscr{H}(\u,\xi,\U,\p,\varphi):= \mathscr{L}(\u,\xi,\U) + \langle\p, \mathcal{N}_1(\u,\xi,\U) \rangle+ \langle\varphi, \mathcal{N}_2(\u,\xi)\rangle,$$ where $\mathcal{N}_1$ and $\mathcal{N}_2$ are defined in (\ref{n1n2}).
	Thus the Pontryagin minimum principle states that 
	\begin{eqnarray}\label{346}
	\mathscr{H}(\u^*(t),\xi^*(t),\U^*(t),\p(t),\varphi(t)) \leq \mathscr{H}(\u^*(t),\xi^*(t),\mathrm{W},\p(t),\varphi(t)),
	\end{eqnarray}
	for all $\mathrm{W} \in \H^{-1}(\Omega)$ and a.e. $t\in[0,T]$. Sometimes, it is obtained in the following integral form also (see \eqref{4.51}):
	\begin{eqnarray}
	\int_0^T	\mathscr{H}(\u^*(t),\xi^*(t),\U^*(t),\p(t),\varphi(t))\d t \leq \int_0^T\mathscr{H}(\u^*(t),\xi^*(t),\mathrm{W},\p(t),\varphi(t))\d t. 
	\end{eqnarray}
	From \eqref{346}, we infer that the following minimum principle is satisfied by the optimal triplet $(\u^*,\xi^*,\U^*)\in\mathscr{A}_{\mathrm{ad}}$ obtained in Theorem \ref{optimal}:
	\begin{equation}\label{pm}
	\frac{1}{2}\|\U^*(t)\|^2_{\mathbb{H}^{-1}}+\langle\p(t),\U^*(t) \rangle \leq \frac{1}{2}\|\mathrm{W}\|^2_{\mathbb{H}^{-1}}+\langle \p(t),\mathrm{W} \rangle,
	\end{equation} 
	for all  $\mathrm{W}\in \H^{-1}(\Omega),$ and a.e. $t\in[0,T]$. For $\mathscr{U}_{\mathrm{ad}}=\mathrm{L}^2(0,T;\H^{-1}(\Omega))$, from \eqref{pm}, we see that $-\p \in \partial \frac{1}{2}\|{\U^*}(t)\|^2_{\H^{-1}}$, where $\partial$ denotes the subdifferential.
	Since $\frac{1}{2}\|\cdot\|^2_{\H^{-1}}$ is Gateaux differentiable,  the subdifferential consists of a single point and it follows that
	\begin{equation*}
	-\p(t) = (-\Delta)^{-1}\U^*(t)\in \H_0^1(\Omega), \ \text{ a.e. }\  t \in [0,T].
	\end{equation*}
	The optimal control is given by $\U^*(t)=\Delta\p(t)\in\H^{-1}(\Omega)$, a.e. $t\in[0,T]$. 
	\vskip 0.2 cm
	(2). The problems involving admissible  control class $\mathscr{U}_{\mathrm{ad}}$ as a closed convex subset of $\mathrm{L}^2(0,T;\L^2(\Omega))$ and incorporating state constraints has been addressed in the work \cite{MTM2}. 
	\vskip 0.2 cm
	{(3). Let us now consider the cost functional as 
		\begin{equation}\label{dcost}
		\begin{aligned}
		\mathscr{J}(\u,\xi,\U) &:= \frac{1}{2} \int_0^Te^{-\widetilde{\lambda} t}\left( \|\u(t)-\u_d(t)\|^2_{\mathbb{L}^2}+\|\nabla\u(t)\|_{\L^2}^2+ \|\xi(t) -\xi_d(t)\|^2_{\mathrm{L}^2}\right) \d t\\&\quad+ \frac{1}{2} \int_0^Te^{-\widetilde{\lambda} t}\|\U(t)\|^2_{\mathbb{H}^{-1}}\d t,
		\end{aligned}
		\end{equation}
		where $\widetilde{\lambda}>0$ is the discount rate. Thus, for the Lagrangian $$\mathscr{L}(\u,\xi,\U) := \frac{1}{2} e^{-\widetilde{\lambda} t}(\|\u-\u_d\|^2_{\mathbb{L}^2} + \|\xi - \xi_d\|^2_{\mathrm{L}^2}+\|\nabla\u\|_{\L^2}^2+ \|\U\|^2_{\mathbb{H}^{-1}}), $$ and for an optimal triplet $(\u^*,\xi^*,\U^*)\in\mathscr{A}_{\mathrm{ad}},$ the Pontryagin minimum principle becomes (see \eqref{346} and \eqref{pm})
		\begin{equation}\label{pm1}
		\frac{1}{2}e^{-\widetilde{\lambda} t}\|\U^*(t)\|^2_{\mathbb{H}^{-1}}+\langle\p(t),\U^*(t) \rangle \leq \frac{1}{2}e^{-\widetilde{\lambda} t}\|\mathrm{W}\|^2_{\mathbb{H}^{-1}}+\langle \p(t),\mathrm{W} \rangle,
		\end{equation} 
		for all  $\mathrm{W}\in \H^{-1}(\Omega),$ and a.e. $t\in[0,T]$. In \eqref{pm1}, the adjoint variable $(\p(\cdot),\varphi(\cdot))$ satisfies the following system: 
		\begin{eqnarray}\label{dadj}
		\left\{
		\begin{aligned}
		-\frac{\partial \p(t)}{\partial t}+\widetilde{\A}\p(t)+\B'(\u(t))\p (t)-h\nabla\varphi(t)&= e^{-\widetilde{\lambda} t}(\u(t)-\u_d(t))\\&\quad-e^{-\widetilde{\lambda} t}\Delta\u(t), \ \text{ in }\ \H^{-1}(\Omega),\\
		-\frac{\partial \varphi(t) }{\partial t}  -\mathrm{div\ }\p(t)&= e^{-\widetilde{\lambda} t}(\xi(t) - \xi_d(t)),   \ \text{ in }\ \mathrm{L}^2(\Omega), \\  (\p(T),\varphi(T))&=(\mathbf{0},0),
		\end{aligned}
		\right.
		\end{eqnarray}
		for a.e. $t\in[0,T]$. 
		For $\mathscr{U}_{\mathrm{ad}}=\mathrm{L}^2(0,T;\H^{-1}(\Omega))$, following similarly as in the proof of Theorem \ref{main}, we obtain the optimal control as 
		\begin{align}
		\U^*(t)=e^{\widetilde{\lambda}t}\Delta\p(t)\in \H^{-1}(\Omega),  \text{ a.e. }  t \in [0,T].
		\end{align}
		Such discount factors appear in infinite horizon optimal control problems (dynamic programming methods, cf. Chapter 6.6, \cite{LiYo}).}
\end{remark}

\subsection{Uniqueness of optimal control in small time interval} In this subsection, we show the uniqueness of optimal control in small time interval for the optimal control problem \eqref{control problem}. Remember that the control to state mapping is nonlinear and getting a global in time unique optimal control is difficult. Thus, we are looking for a time $T$ such that this time ensures uniqueness of optimal control. If we choose the final time $T$ to be  sufficiently small, then the state equation for $(\u,\xi)$ differ from the corresponding linearized problem $(\mathfrak{w},\eta)$ slightly only.  In this case, the linearized state equation corresponding to $(\mathfrak{w},\eta)$  produces a strictly convex cost functional and the corresponding optimal control is unique.
\begin{theorem}\label{thm3.8}
	Let $(\u^*,\xi^*,\U^*)\in\mathscr{A}_{\mathrm{ad}}$ be an  \emph{optimal triplet} for the problem \eqref{control problem}. Then if the final time $T$ is sufficiently small such that 
	\begin{align}\label{347}
	\frac{C_{\Omega}}{\alpha^2} \left(\frac{r^2}{\lambda^2}C_{\Omega}K_T +1\right) e^{ \frac{C_{\Omega}T}{\alpha}\left(2+\mu^2+M\right)}<1,
	\end{align}
	where $K_T$ is defined in \eqref{40a},	is satisfied, then the optimal triplet $(\u^*,\xi^*,\U^*)\in\mathscr{A}_{\mathrm{ad}}$ obtained in the Theorem \ref{optimal}  is unique. 
\end{theorem}
\begin{proof}
	Let us assume that there exists an another optimal triplet $(\widetilde{\u},\widetilde{\xi},\widetilde{\U})\in\mathscr{A}_{\mathrm{ad}}$. 
	From Theorem \ref{main}, we know that  $\U^*(t)=\Delta\p^*(t)\in\H^{-1}(\Omega)$, a.e. $t\in[0,T]$, and $\widetilde{\U}(t)=\Delta\widetilde{\p}(t)\in\H^{-1}(\Omega)$, a.e. $t\in[0,T]$. Note also that $(\p^*,\xi^*)$ and $(\widetilde{\p},\widetilde{\xi})$ satisfy the adjoint system \eqref{adj} with the forcing $(\u^*-\u_d-\Delta\u^*,\xi^*-\xi_d)$ and $(\widetilde{\u}-\u_d-\Delta\widetilde{\u},\widetilde{\xi}-\xi_d)$, respectively. Then, for a.e. $t\in[0,T]$, we have 
	\begin{align}
	\|\U^*(t)-\widetilde{\U}(t)\|_{\H^{-1}}&=\|\Delta\p^*(t)-\Delta\widetilde{\p}(t)\|_{\H^{-1}}= \|\p^*(t)-\widetilde{\p}(t)\|_{\H_0^1}.
	\end{align}
	Thus, we obtain
	\begin{align}\label{3.46}
	\int_0^T\|\U^*(t)-\widetilde{\U}(t)\|_{\H^{-1}}^2\d t= \int_0^T\|\p^*(t)-\widetilde{\p}(t)\|_{\H_0^1}^2\d t.
	\end{align}
	The quantity on the right hand side of the inequality \eqref{3.46} can be estimated similarly as in  \eqref{40}, as the system satisfied by $(\widetilde{\p}-\p^*,\widetilde{\xi}-\xi^*)$ is given by 
	\begin{eqnarray}\label{3.47}
	\left\{
	\begin{aligned}
	-\frac{\partial \p(t)}{\partial t}+\widetilde{\A}\p(t)+(\B'(\widetilde{\u}(t))\widetilde{\p}(t)&-\B'(\u^*(t))\p^* (t))-h\nabla\varphi(t)\\&= \u(t)-\Delta\u(t), \ \text{ in }\ \H^{-1}(\Omega),\\
	-\frac{\partial \varphi(t) }{\partial t}  -\mathrm{div\ }\p(t)&= \xi(t),   \ \text{ in }\ \mathrm{L}^2(\Omega), \\ (\p(T),\varphi(T))&=(\mathbf{0},0), 
	\end{aligned}
	\right.
	\end{eqnarray}
	where $\p=\widetilde{\p}-\p^*$,  $\varphi=\widetilde{\varphi}-\varphi^*$, $\u=\widetilde\u-\u^*$ and $\xi=\widetilde\xi-\xi^*$. Taking inner product with $\p(\cdot)$ to the first equation in \eqref{3.47} and $\varphi(\cdot)$ to the second equation in \eqref{3.47}, and then adding them together, we find 
	\begin{align}\label{350}
	&-\frac{1}{2}\frac{\d}{\d t}\left(\|\p(t)\|_{\L^2}^2+\|\varphi(t)\|_{\mathrm{L}^2}^2\right)+\alpha\|\nabla\p(t)\|_{\L^2}^2+(\B'(\widetilde\u(t))\p(t),\p(t))_{\L^2}\nonumber\\&=- ((\B'(\widetilde{\u}(t))-\B'(\u^*(t)))\p^*(t),\p(t))_{\mathbb{L}^2}-\langle h\nabla\varphi(t),\p(t)\rangle+\langle\u(t)-\Delta\u(t),\p(t)\rangle\nonumber\\&\quad+(\mathrm{div\ }\p(t),\varphi(t))_{\mathrm{L}^2}+(\xi(t),\varphi(t))_{\mathrm{L}^2}\nonumber\\&\leq\|\B'(\widetilde{\u}(t))-\B'(\u^*(t))\|_{\L^4}\|\p^*(t)\|_{\L^4}\|\p(t)\|_{\L^2}+(\varphi(t),\mathrm{div}(h\p(t)))_{\mathrm{L}^2}\nonumber\\&\quad+(\u(t),\p(t))_{\L^2}+(\nabla\u(t),\nabla\p(t))_{\L^2}+\|\mathrm{div\ }\p(t)\|_{\mathrm{L}^2}\|\varphi(t)\|_{\mathrm{L}^2}+\|\xi(t)\|_{\mathrm{L}^2}\|\varphi(t)\|_{\mathrm{L}^2}\nonumber\\&\leq \frac{2rC_{\Omega}}{\lambda}\|\u(t)\|_{\L^4}\|\p^*(t)\|_{\L^4}\|\nabla\p(t)\|_{\L^2}+M\|\varphi(t)\|_{\mathrm{L}^2}\|\p(t)\|_{\L^2}\nonumber\\&\quad+\sqrt{2}\mu\|\varphi(t)\|_{\mathrm{L}^2}\|\nabla\p(t)\|_{\L^2}+\|\u(t)\|_{\L^2}\|\p(t)\|_{\L^2}+\|\nabla\u(t)\|_{\L^2}\|\nabla\p(t)\|_{\L^2}\nonumber\\&\quad+\sqrt{2}\|\nabla\p(t)\|_{\L^2}\|\varphi(t)\|_{\mathrm{L}^2}+\|\xi(t)\|_{\mathrm{L}^2}\|\varphi(t)\|_{\mathrm{L}^2}\nonumber\\&\leq\frac{\alpha}{2}\|\nabla\p(t)\|_{\L^2}^2+\frac{r^2C_{\Omega}}{\lambda^2}\|\u(t)\|_{\L^4}^2\|\p^*(t)\|_{\L^4}^2+\left[\frac{C_{\Omega}M^2}{\alpha}+\frac{6(\mu^2+1)}{\alpha}+\frac{1}{2}\right]\|\varphi(t)\|_{\mathrm{L}^2}^2\nonumber\\&\quad+\frac{C_{\Omega}}{\alpha}\|\u(t)\|_{\L^2}^2+\frac{3}{\alpha}\|\nabla\u(t)\|_{\L^2}^2+\frac{1}{2}\|\xi(t)\|_{\mathrm{L}^2}^2. 
	\end{align}
	Integrating the inequality \eqref{350} from $t$ to $T$, we obtain
	\begin{align}\label{3.49}
	&\|\p(t)\|_{\mathbb{L}^2}^2+\|\varphi(t)\|_{\mathrm{L}^2}^2+\alpha\int_t^T\|\nabla\p(s)\|_{\mathbb{L}^2}^2\d s+2\int_t^T\||\widetilde\u(s)+\w^0(s)|^{\frac{1}{2}}\p(s)\|_{\mathbb{L}^2}^2\d s\nonumber\\&\leq \frac{r^2C_{\Omega}}{\lambda^2}\int_t^T\|\u(s)\|_{\L^4}^2\|\p^*(s)\|_{\L^4}^2\d s+\left[\frac{C_{\Omega}M^2}{\alpha}+\frac{12(\mu^2+1)}{\alpha}+1\right]\int_t^T\|\varphi(s)\|_{\mathrm{L}^2}^2\d s\nonumber\\&\quad+\frac{(C_{\Omega}+6)}{\alpha}\int_t^T\|\nabla\u(s)\|_{\L^2}^2\d s+\int_t^T\|\xi(s)\|_{\mathrm{L}^2}^2\d s. 
	\end{align}
	An application of Gronwall's inequality in \eqref{3.49} yields 
	\begin{align}\label{352}
	&\|\p(t)\|_{\mathbb{L}^2}^2+\|\varphi(t)\|_{\mathrm{L}^2}^2\nonumber\\&\leq \bigg\{\frac{r^2C_{\Omega}}{\lambda^2}\left(\int_0^T\|\u(t)\|_{\L^4}^4\d t\right)^{1/2}\left(\int_0^T\|\p^*(t)\|_{\L^4}^4\d t\right)^{1/2}+\frac{(C_{\Omega}+6)}{\alpha}\int_0^T\|\nabla\u(t)\|_{\L^2}^2\d t\nonumber\\&\quad+\int_0^T\|\xi(t)\|_{\mathrm{L}^2}^2\d t\bigg\}\exp\left\{\left[\left(\frac{C_{\Omega}M^2}{\alpha}+\frac{12(\mu^2+1)}{\alpha}\right)+1\right]T\right\},
	\end{align}
	for all $t\in[0,T]$. Using \eqref{352} in \eqref{3.49} and then taking $t=0$, we get 
	\begin{align}\label{3.51}
	\int_0^T\|\nabla\p(s)\|_{\mathbb{L}^2}^2\d s&\leq\frac{1}{\alpha} \Bigg\{\frac{r^2C_{\Omega}}{\lambda^2}\left(\int_0^T\|\u(t)\|_{\L^4}^4\d t\right)^{1/2}\left(\int_0^T\|\p^*(t)\|_{\L^4}^4\d t\right)^{1/2}\nonumber\\&\quad+\frac{(C_{\Omega}+6)}{\alpha}\int_0^T\|\nabla\u(t)\|_{\L^2}^2\d t+\int_0^T\|\xi(t)\|_{\mathrm{L}^2}^2\d t\Bigg\}\nonumber\\&\quad\times\exp\left\{2\left[\left(\frac{C_{\Omega}M^2}{\alpha}+\frac{12(\mu^2+1)}{\alpha}\right)+1\right]T\right\}\nonumber\\&\leq\frac{C_{\Omega}}{\alpha^2} \left(\frac{r^2}{\lambda^2}C_{\Omega}K_T +1\right) \left[\int_0^T\|\U^*(t)-\widetilde\U(t)\|_{\H^{-1}}^2\d t\right]e^{ \frac{C_{\Omega}T}{\alpha}\left(2+\mu^2+M\right)},
	\end{align}
	where we used \eqref{325a} and \eqref{40a} also. Combining \eqref{3.46} and \eqref{3.51}, it can be easily seen that 
	\begin{align}
	&	\int_0^T\|\U^*(t)-\widetilde{\U}(t)\|_{\H^{-1}}^2\d t \nonumber\\&\leq \frac{C_{\Omega}}{\alpha^2} \left(\frac{r^2}{\lambda^2}C_{\Omega}K_T +1\right) \left[\int_0^T\|\U^*(t)-\widetilde\U(t)\|_{\H^{-1}}^2\d t\right]e^{ \frac{C_{\Omega}T}{\alpha}\left(2+\mu^2+M\right)},
	\end{align}
	where $K_T$ is defined in \eqref{40a}. 	 Now, if we choose  time $T$ sufficiently small such that \eqref{347} is satisfied,	then we get that $\|\U^*-\widetilde{\U}\|_{\mathrm{L}^2(0,T;\H^{-1}(\Omega))}=0$. Thus, we obtain $\U^*(t)=\widetilde{\U}(t)$, a.e.x $t\in[0,T]$, for sufficiently small $T$.  This gives the uniqueness of the optimal control up to the time $T$ satisfying \eqref{347}.
\end{proof}

\subsection{Data assimilation problem} Let us now consider a problem similar to the data assimilation problems of meteorology. In the data assimilation problems arising from meteorology, determining the accurate initial data for the future predictions is an important step. This motivates us to consider a similar problem for the tidal dynamics.  We pose a optimal data initialization problem, where we find the unknown optimal initial data by minimizing suitable cost functional  subject to the tidal dynamics equations (see \cite{sritharan} for the case of incompressible Navier-Stokes equations and \cite{BDM2} for Cahn-Hillard-Navier-Stokes equations).  

Let us formulate the initial data optimization problem as finding an optimal initial velocity  $\U\in\L^2(\Omega)$ such that $(\u, \xi, \U) $ satisfies the following system:
\begin{equation}\label{11111}
\left\{
\begin{aligned}
\frac{\partial \mathbf{u}(t)}{\partial t}+\A\mathbf{u}(t)+\B(\mathbf{u}(t))+g\nabla \xi (t)&=\mathbf{f}(t),\ \text{ in }\ \H^{-1}(\Omega),\\
\frac{\partial \xi(t)}{\partial t}+\mathrm{div}(h \mathbf{u}(t))&=0, \ \text{ in }\ \mathrm{L}^2(\Omega),\\
\mathbf{u}(0)=\U\in\L^2(\Omega),\ \xi(0)&=\xi_0\in\mathrm{L}^2(\Omega),
\end{aligned}
\right.
\end{equation}
for a.e. $t\in[0,T]$ and minimizes the cost functional 
\begin{equation}\label{cost4}
\begin{aligned}
\mathscr{J}(\u,\xi,\U)&:=  \frac{1}{2} \|\U\|^2_{\L^2}+\frac{1}{2} \int_0^T \|\u(t) - \u_M(t)\|^2_{\L^2} \d t+ \frac{1}{2} \int_0^T \|\xi(t) -\xi_M(t)\|^2_{\mathrm{L}^2} \d t\\ &\quad+\frac{1}{2}\|\u(T)-\u_M^f\|^2_{\mathbb{L}^2}+\frac{1}{2}\|\xi(T)-\xi_M^f\|^2_{\mathrm{L}^2},
\end{aligned}
\end{equation}
where $\u_M$ is the measured average velocity of the fluid and $\xi_M$  is the measured elevation, $\u_M^f$ and $\xi_M^f$ are measured velocity and elevation at time $T$, respectively. In order to make the cost functional given in \eqref{cost4} meaningful, we assume that 
\begin{align}\label{um}\u_M\in\mathrm{L}^2(0,T;\L^2(\Omega)),\ \xi_M\in\mathrm{L}^2(0,T;\mathrm{L}^2(\Omega)),\ \u_M^f\in\L^2(\Omega)\ \text{ and }\ \xi_M^f\in\mathrm{L}^2(\Omega).\end{align}
In this context, we take the set of admissible control class, $\mathscr{U}_{\mathrm{ad}}=\L^2(\Omega)$.  Also, the  \emph{admissible class} $\mathscr{A}_{\mathrm{ad}}$ consists of all triples $(\u,\xi,\U)$ such that the set of states $(\u,\xi)$ is the unique weak solution of the tidal dynamics system \eqref{11111} with the control $\U \in \mathscr{U}_{ad} $.  We formulate the optimal control problem as:
\begin{align} \label{IOCP}
\min_{(\u,\xi,\U) \in \mathscr{A}_{\mathrm{ad}} } \mathscr{J}(\u,\xi,\U).
\end{align}

The next theorem provides the existence of  an optimal triplet $(\u^*,\xi^*,\U^*)$ for the problem \eqref{IOCP}. A proof of the Theorem can be established similar to that of Theorem \ref{optimal}.
\begin{theorem}[Existence of an optimal triplet]\label{optimal1}
	Let the initial data $\xi_0\in\mathrm{L}^2(\Omega)$ be given and let $\U\in\mathscr{U}_{\mathrm{ad}}$. Then there exists at least one triplet  $(\u^*,\xi^*,\U^*)\in\mathscr{A}_{\mathrm{ad}}$  such that the functional $ \mathscr{J}(\u,\xi,\U)$ given in \eqref{cost4} attains its minimum at $(\u^*,\xi^*,\U^*)$, where $(\u^*,\xi^*)$ is the unique weak solution of the system \eqref{11111}  with the initial data control $\U^*\in\mathscr{U}_{\mathrm{ad}}$.
\end{theorem}
Formally, Pontryagin's minimum principle gives
\begin{align*}
\frac{1}{2}\|\U^*\|^2_{\L^2} +( \p(0),\U^* )_{\L^2}\leq \frac{1}{2}\|\mathrm{W}\|^2_{\L^2} + (\p(0),\mathrm{W})_{\L^2},
\end{align*}
for all $\mathrm{W} \in \mathscr{U}_{\mathrm{ad}}$. Since $\mathscr{U}_{\mathrm{ad}}=\L^2(\Omega)$ and $\frac{1}{2}\|\cdot\|_{\L^2}^2$ is Gateaux differentiable, then the optimal control is given by $\U^*=-\p(0)$, where $(\p,\varphi)$ is the unique weak solution of the following adjoint system:  
\begin{eqnarray}\label{adj3}
\left\{
\begin{aligned}
-\frac{\partial \p(t)}{\partial t}+\widetilde{\A}\p(t)+\B'(\u(t))\p (t)-h\nabla\varphi(t)&= \u(t)-\u_M(t), \ \text{ in }\ \H^{-1}(\Omega),\\
-\frac{\partial \varphi(t) }{\partial t}  -\mathrm{div\ }\p(t)&= \xi(t) - \xi_M(t),   \ \text{ in }\ \mathrm{L}^2(\Omega), \\ (\p(T), \varphi(T))=(\u(T)-\u_M^f, \varphi(T)-\u_M^f)&\in\L^2(\Omega)\times\mathrm{L}^2(\Omega),
\end{aligned}
\right.
\end{eqnarray}
for a.e. $t\in[0,T]$. A similar calculation as in the proof of Theorem  \ref{thm3.4} yields the existence of a unique weak solution to the system (\ref{adj3}) such that  \begin{align}\label{348}(\p,\varphi)\in (\C([0,T];\L^2(\Omega))\cap\mathrm{L}^2(0,T;\mathbb{H}^1_0(\Omega)))\times (\C([0,T];\mathrm{L}^2(\Omega))).\end{align}
Using the continuity of $\p(\cdot)$ in time  at $t=0$ in $\L^2(\Omega)$, we know that $\p(0)\in\L^2(\Omega)$ and hence we get   \begin{align}\label{in}\U^*=-\p(0)\in\mathscr{U}_{\mathrm{ad}}=\L^2(\Omega).\end{align}
Thus, we have the following Theorem, which can be proved in a similar way as that of  Theorem \ref{main} with some obvious modifications. 

\begin{theorem}[Optimal initial control]\label{data}
	Let $(\u^*,\xi^*,\U^*)\in\mathscr{A}_{\mathrm{ad}}$ be an optimal triplet. Then there exists a unique weak solution $(\p,\varphi)$ of the adjoint system \eqref{adj3} satisfying \eqref{348} such that the optimal control is obtained via \eqref{in}. 
\end{theorem}

\medskip\noindent
{\bf Acknowledgements:} M. T. Mohan would  like to thank the Department of Science and Technology (DST), India for Innovation in Science Pursuit for Inspired Research (INSPIRE) Faculty Award (IFA17-MA110). The author sincerely would like to thank the reviewers for their  valuable comments and suggestions.


\begin{thebibliography}{99}
	\bibitem{FART}  F. Abergel,  R. Temam,  On some control problems in fluid mechanics, \emph{Theoretical and Computational Fluid Dynamics}, {\bf 1 }  (1990), 303--325.
	
	
	\bibitem{AMM}  P. Agarwal,  U. Manna and  D. Mukherjee, Stochastic control of tidal dynamics equation with L\'evy noise, \emph{Applied Mathematics and Optimization}, Appl. Math. Optim. {\bf 79}(2) (2019), 327--396. 
	
	\bibitem{VIA}	V. I. Agoshkov and  E. A. Botvinovsky,
	Numerical solution of a hyperbolic-parabolic system by splitting methods and optimal control approaches, 
	\emph{Comput. Methods Appl. Math.}, {\bf 7}(3) (2007), 193--207.
	
	
	\bibitem{VB}  V. Barbu, \emph{Analysis and Control of Nonlinear Infinite Dimensional Systems}, Mathematics in Science and Engineering, vol. 190, Academic Press Inc., Boston, MA, 1993.
	
	
	
	\bibitem{NRCB}   N. R. C. Birkett  and  N. K. Nichols,  Optimal control problems in tidal power generation, \emph{ Industrial numerical analysis}, 53--89, Oxford Sci. Publ., Oxford Univ. Press, New York,1986. 
	
	\bibitem{BDM}	T. Biswas, S. Dharmatti and M. T. Mohan, Pontryagin maximum principle and second order optimality conditions for optimal control problems governed by 2D nonlocal Cahn-Hilliard-Navier-Stokes equations, \emph{Analysis (Berlin)}, {\bf 40}(3), 127--150.
	
	\bibitem{BDM2}	T. Biswas, S. Dharmatti and M. T. Mohan, Maximum principle and data assimilation problem for the optimal control problems governed by 2D nonlocal Cahn-Hilliard-Navier-Stokes equations,  \emph{Journal of Mathematical Fluid Mechanics}, {\bf 22}, Article number: 34, 2020.
	
	
	
	
	\bibitem{PGC} 	P. G. Ciarlet, \emph{Linear and Nonlinear Functional Analysis with Applications}, SIAM Philadelphia, 2013. 
	
	
	
	
	
	
	\bibitem{SDMTS}	S. Doboszczak,  M. T. Mohan,  S. S. Sritharan,  Existence of optimal controls for compressible viscous flow, \emph{ Journal of Mathematical Fluid Mechanics}, {\bf 20}(1) (2018), 199--211.
	
	
	\bibitem{evans} L. C. Evans, \emph{Partial differential equations}, Grad. Stud. Math., vol. {\bf 19}, Amer. Math. Soc., Providence, RI, 1998.
	
	
	\bibitem{EI} I. Ekeland, and T. Turnbull, \emph{Infinite-dimensional Optimization and Convexity}, The University of Chicago press, Chicago and London, 1983. 
	
	
	
	
	\bibitem{fursikov} A. V. Fursikov, Optimal control of distributed systems: Theory and applications, \emph{American Mathematical Society}, Rhode Island (2000).
	
	
	
	
	
	
	\bibitem{GaG} G. Galilei,  \emph{Dialogue Concerning the Two Chief World Systems}, 1632.
	
	\bibitem{RGG} R. G. Gordeev, The existence of a periodic solution in tide dynamic problem, \emph{Journal of Soviet Mathematics}, {\bf 6}(1) (1976), 1--4.
	
	\bibitem {gunzburger} M. D. Gunzburger,  Perspectives in Flow Control and Optimization,  \emph{SIAM's Advances in Design and Control series}, Philadelphia (2003).
	
	
	
	\bibitem{HSMB}	A. Haseena, M. Suvinthra, M. T.  Mohan and K. Balachandran, Moderate deviations for stochastic tidal dynamics equation with multiplicative noise, Published online in \emph{Applicable Analysis}, 2020,  DOI: 10.1080/00036811.2020.1781827. 
	
	\bibitem{IpM} V. M.	Ipatova,  Solvability of a tide dynamics model in adjacent seas, \emph{Russ. J. Numer. Anal. Math. Modelling}, {\bf 20} (1) (2005), 67--79.
	
	\bibitem{BAK} B. A. Kagan, Hydrodynamic Models of Tidal Motions in the Sea (Russian), Gidrometeoizdat, Leningrad, 1968.
	
	
	\bibitem{OAL}	O. A. Ladyzhenskaya, \emph{The Mathematical Theory of Viscous Incompressible Flow}, Gordon and Breach, New York, 1969.
	
	\bibitem{lions} J.-L. Lions, \emph{Optimal Control of Systems Governed by Partial Differential Equations}, Springer, 1971. 
	
	\bibitem{LiYo}	X. Li and J. Yong, \emph{Optimal Control Theory for Infinite Dimensional Systems}, Birkhauser Boston, 1995.
	
	
	
	
	\bibitem{MK} G. I. Marchuk,  and B.A. Kagan,  \emph{Ocean tides: Mathematical models and numerical experiments},
	Pergamon Press, Elmsford, NY, 1984.
	
	\bibitem{MK1} G. I. Marchuk,  and B. A. Kagan,  \emph{Dynamics of Ocean Tides}, Kluwer Academic Publishers,
	Dordrecht/Boston/London, 1989.
	
	\bibitem{Manna}
	U. Manna, J. L. Menaldi, and S. S. Sritharan, \textit{Stochastic analysis of tidal dynamics equation}, Infinite Dimensional Stochastic Analysis, (2008), 90-113.
	
	\bibitem{MTM} M. T. Mohan, On the two dimensional tidal dynamics system: stationary solution and stability,  \emph{Applicable Analysis}, {\bf 99}(10), 1795--1826. 
	
	\bibitem{MTM1} M. T. Mohan, 	Dynamic programming and feedback analysis of the two dimensional tidal dynamics system, Accepted in \emph{ESAIM: Control, Optimisation and Calculus of Variations}, 2020, https://doi.org/10.1051/cocv/2020025. 
	
	\bibitem{MTM2} 	M. T. Mohan,  Necessary conditions for distributed optimal control of two dimensional tidal dynamics system with state constraints, \emph{Submitted}. 
	
	
	\bibitem{RMo} 	R. Mosetti, Optimal control of sea level in a tidal basin by means of the Pontryagin maximum principle, \emph{Applied Mathematical Modelling},  {\bf 9}(5) (1985), 321--324. 
	
	
	\bibitem{IN} I. Newton,  \emph{Philosophiae Naturalis Principia Mathematica}, 1687.
	
	\bibitem{PJ}  J. Pedlosky, \emph{Geophysical Fluid Dyanmics I, II}, Springer, Heidelberg, 1981.
	
	\bibitem{raymond} J. P. Raymond, Optimal control of partial differential equations. Universit\'e Paul Sabatier, \emph{Lecture Notes}, 2013.
	
	
	\bibitem{RyR} S. C. Ryrie and D. T. Bickley, Optimally controlled hydrodynamics for tidal power in the Severn Estuary, \emph{Appl. Math. Modelling}, {\bf 9} (1985), 1--10. 
	
	\bibitem{RyR1} S. C. Ryrie, An optimal control model of tidal power generation, \emph{Appl. Math. Modelling}, {\bf 19} (1985), 123--126. 
	
	
	
	
	\bibitem{SJ}  J. Simon,  Compact sets in the space $\mathrm{L}^p(0,T;\mathrm{B})$. Annali di Matematica Pura ed Applicata. {\bf 146}, 65--96 (1986).
	
	
	\bibitem{EMS} E. M. Stein, \emph{Singular Integrals and Differentiability Properties of Functions}, Princeton University Press, 1970. 
	
	\bibitem{sritharan} S. S. Sritharan, \emph{Optimal control of viscous flow}, SIAM Frontiers in Applied Mathematics, Philadelphia. Society for Industrial and Applied Mathematics, 1998.
	
	\bibitem{SSKB} M. Suvinthra, S. S. Sritharan and K. Balachandran,
	{\it Large deviations for stochastic tidal dynamics equations},
	Communication on Stochastic Analysis, {\bf 9} (4) (2015), 477--502.
	
	
	
	
	
	
	\bibitem{HW}   	H. Whitney, Analytic extension of differentiable functions defined in closed sets, \emph{Trans. Amer. Math. Soc.}, {\bf 36} (1934),  63--89. 
	
	\bibitem{ZYJM} Z.  Yanga and J. M. Hamrickb, Optimal control of salinity boundary condition in a tidal model using a variational inverse method, \emph{Estuarine, Coastal and Shelf Science}, {\bf 62} (2005), 13--24.
	
	
	
	\bibitem{yin}
	H. Yin,
	\textit{Stochastic analysis of backward tidal dynamics equation}, Communications on Stochastic Analysis, \textbf{5} (2011), 745-768.
	
	
	
	
\end{thebibliography}
\end{document}